\UseRawInputEncoding
\documentclass[12pt, reqno]{amsart}
\usepackage[margin=1in]{geometry}
\usepackage{amssymb,latexsym,amsmath,amscd,amsfonts}
\usepackage{latexsym}
\usepackage[mathscr]{eucal}
\usepackage{bm}
\usepackage{mathptmx}
\usepackage{amssymb}
\usepackage{amsthm}
\usepackage{dcolumn}
\usepackage[all]{xy}
\usepackage{enumitem}
\usepackage[utf8]{inputenc}

\def \qed {\hfill \vrule height6pt width 6pt depth 0pt}
\def\textmatrix#1&#2\\#3&#4\\{\bigl({#1 \atop #3}\ {#2 \atop #4}\bigr)}
\def\dispmatrix#1&#2\\#3&#4\\{\left({#1 \atop #3}\ {#2 \atop #4}\right)}
\newcommand{\beg}{\begin{equation}}
	\newcommand{\eeg}{\end{equation}}
\newcommand{\ben}{\begin{eqnarray*}}
	\newcommand{\een}{\end{eqnarray*}}

\newtheorem{thm}{Theorem}[section]
\newtheorem{cor}[thm]{Corollary}
\newtheorem{lem}[thm]{Lemma}

\newtheorem{prop}[thm]{Proposition}
\numberwithin{equation}{section} \theoremstyle{definition}
\newtheorem{defn}[thm]{Definition}

\newtheorem{eg}[thm]{Example}

\newcommand{\C}{\mathbb{C}}
\newcommand{\A}{\mathbb A}
\newcommand{\B}{\mathbb{B}}
\newcommand{\CA}{\overline{\mathbb A}}
\newcommand{\D}{\mathbb{D}}
\newcommand{\T}{\mathbb{T}}
\newcommand{\N}{\mathbb{N}}
\newcommand{\Z}{\mathbb{Z}}

\newcommand{\KS}{\mathcal{K}}
\newcommand{\HS}{\mathcal{H}}

\newcommand{\BC}{\overline{\mathbb{B}}_2}
\newcommand{\DC}{\overline{\mathbb{D}}}
\newcommand{\UT}{\underline{T}}
\newcommand{\la}{\left \langle}
\newcommand{\ra}{\right \rangle}

\def\textmatrix#1&#2\\#3&#4\\{\bigl({#1 \atop #3}\ {#2 \atop #4}\bigr)}
\def\dispmatrix#1&#2\\#3&#4\\{\left({#1 \atop #3}\ {#2 \atop #4}\right)}

\begin{document}

\title[Dilation on $\mathbb A_r$ and von Neumann's inequality on certain varieties in $\mathbb{B}_2$]{Dilation on an annulus and von Neumann's inequality on certain varieties in the biball}

\author[Pal and Tomar]{SOURAV PAL AND NITIN TOMAR}

\address[Sourav Pal]{Mathematics Department, Indian Institute of Technology Bombay,
	Powai, Mumbai - 400076, India.} \email{sourav@math.iitb.ac.in, souravmaths@gmail.com}

\address[Nitin Tomar]{Mathematics Department, Indian Institute of Technology Bombay, Powai, Mumbai-400076, India.} \email{tnitin@math.iitb.ac.in, tomarnitin414@gmail.com}		

\keywords{$\A_r$-contraction, Rational dilation, Spectral set, von Neumann set, Principal variety}	

\subjclass[2020]{47A20, 47A25, 32A60}	


\begin{abstract}
A Hilbert space operator $T$ is said to be an \textit{$\A_r$-contraction} if  the closure of the annulus 
\[
\A_r=\{z \in \mathbb C \ : \ r<|z|<1\}  \qquad (0<r<1)	
	\]
		is a spectral set for $T$. An $\A_r$-unitary is a normal operator with spectrum inside the boundary $\partial \A_r$. A celebrated theorem due to Agler states that every $\A_r$-contraction dilates to an $\A_r$-unitary. We give a short and new proof to Agler's theorem by an application of a result due to Dritschel, Jury and McCullough. Let us consider the algebraic variety $Z(q)$ generated by the polynomial $q(z_1, z_2)=z_1z_2-r\slash (1+r^2)$ and its intersection with the biball $\mathbb{B}_2=\{(z_1, z_2) : |z_1|^2+|z_2|^2<1\}$. We prove that an invertible operator $T$ is an $\A_r$-contraction if and only if $Z(q) \cap \overline{\mathbb{B}}_2$ is a spectral set for the operator pair $\kappa(T)$ if and only if $Z(q) \cap \overline{\mathbb{B}}_2$ is a complete spectral set for the operator pair $\kappa(T)$, where 
\[
\kappa(T)=\left(\frac{T}{\sqrt{1+r^2}}, \frac{rT^{-1}}{\sqrt{1+r^2}} \right).
\]		
As an application of this result, we prove independently a few Wold-type decomposition results for $\A_r$-contractions and $\A_r$-isometries. We have from the literature that $\mathcal A_r \subsetneq C_{\alpha} \subsetneq C_{1,r}$, where $\mathcal A_r$ is the set of all $\mathbb A_r$-contractions and 
\begin{align*}
C_\alpha &=\{T: \ \text{$T$ is invertible and} \ \alpha(T^*, T)=-T^{*2}T^2+(1+r^2)T^*T-r^2I \geq 0\},\\
C_{1, r} & =\{T: \ T \text{ is invertible and } \|T\|, \|rT^{-1}\| \leq 1\}.
\end{align*}
We show that $T \in C_\alpha$ if and only if $\kappa(T)$ is a spherical contraction. We determine minimal spectral sets for $C_{\alpha}, \, C_{1,r}$ and a few more associated classes. Then, we find minimal von Neumann sets for analogous classes of operator pairs induced by the map $\kappa$ in terms of the variety $Z(q)$.
	\end{abstract} 
	
	\maketitle
	
			
	\section{Introduction}\label{Intro}
	
	
	\noindent 	Throughout the paper, all operators are bounded linear maps acting on complex Hilbert spaces. For a Hilbert space $\HS$, $\mathcal{B}(\HS)$ is the algebra of operators on $\HS$. We denote by $\C, \D$ and $\T$ the complex plane, the unit disk and the unit circle in the complex plane respectively, with center at the origin. A contraction is an operator with norm at most $1$.
	
\smallskip	
	
	Given a positive real number $K$, a compact set $X \subset \C^n$ is said to be a $K$-\textit{spectral set} for a tuple of commuting operators $\underline T =(T_1, \dots, T_n)$ on  a Hilbert space $\HS$ if the Taylor joint spectrum $\sigma_T(\underline T) \subseteq X$ and von Neumann's inequality
	\begin{equation} \label{eqn:new-001}
	\|f(\underline{T})\| \leq K \ \|f\|_{\infty, X}=K\sup\{|f(\xi)| \ : \ \xi \in X  \}
	\end{equation}
holds for all $f \in Rat(X)$. Here $Rat(X)$ is the algebra of rational functions $f=p \slash q$, where $p, q$ are holomorphic polynomials in $n$ complex variables $z_1, \dots , z_n$ with $q$ having no zeros in $X$. Also, $f(\underline T)=p(\underline T)q(\underline T)^{-1}$, where $q(\underline T)$ is invertible as $q$ has no zeros in $X$. Again, $X$ is said to be a \textit{complete $K$-spectral set} for $\underline T$ if $\sigma_T(\underline T) \subseteq X$ and for every matricial rational function $f=[f_{ij}]_{i, j=1}^m$ with each $f_{ij} \in Rat(X)$,
\begin{equation} \label{eqn:new-002}
	\|f(T_1, \dotsc, T_n)\|=\|[f_{ij}(T_1, \dotsc, T_n)]_{i, j=1}^m \| \leq K \ \|f\|_{\infty,\, X}=K \ \sup\{\|[f_{ij}(\xi)]\| \ : \ \xi \in X  \}.
	\end{equation}
	Also, $X$ is said to be a \textit{spectral set} or \textit{complete spectral set} for $\underline T$ if (\ref{eqn:new-001}) or (\ref{eqn:new-002}) respectively, is satisfied when $K=1$. Moreover, $X$ is said to be a \textit{minimal spectral set} (or a \textit{complete minimal spectral set}) for a set of commuting $n$-tuples of operators $\mathcal T$ if $X$ is a spectral set (or complete spectral set) for every member of $\mathcal T$ and for any proper compact subset $X'$ of $X$, there is $\underline T \in \mathcal T$ such that $X'$ is not a spectral set (or complete spectral set) for $\underline T$. In 1951, von Neumann introduced the notion of spectral set in \cite{v-N} where he profoundly established the fact that an operator is a contraction if and only if the closed unit disk $\overline{\D}$ is a spectral set for it.
	
	\smallskip 
	
	A commuting $n$-tuple of operators $\underline{T}$ having $X$ as a spectral set is said to have a \textit{rational dilation} or a \textit{normal $bX$-dilation} if there exist a Hilbert space $\mathcal{K},$ an isometry $V: \mathcal{H} \to \mathcal{K}$ and a commuting $n$-tuple of normal operators $\underline{N}=(N_1, \dots , N_n)$ on 
	$\mathcal{K}$ with $\sigma_T(\underline{N})$ lying in the distinguished boundary $bX$ of $X$ such that 
	\begin{equation}\label{eq:Rat}
		f(\underline{T})=V^*f(\underline{N})V
	\end{equation}
	for all $f \in Rat(X)$. In other words, $f(\underline{T})=P_\mathcal{H}f(\underline{N})|_\mathcal{H}$ for every $f \in Rat(X)$ when $\mathcal{H}$ is considered a closed subspace of $\mathcal{K}$ and $P_\HS$ is the orthogonal projection of $\mathcal{K}$ onto the space $\mathcal{H}$. The dilation is said to be \textit{minimal} if 
	\[
	\mathcal{K}=\overline{\mbox{span}}\{f(\underline{N})h \ :\ h \in \mathcal{H} \ \& \ f \in Rat(X)\}.
	\]
	Rational dilation is said to succeed on a compact set $X\subset \C^n$ if every commuting $n$-tuple of operators having $X$ as a spectral set admits a normal $bX$-dilation. In 1953, Sz. Nagy \cite{B-Nagy} established the success of rational dilation on the closed unit disk $\overline{\D}$ by dilating every contraction to a unitary. The success or failure of rational dilation on an annulus was unknown for a few decades. In 1985, Agler settled the rational dilation problem for an annulus affirmatively in his seminal paper \cite{Agler}.
	 
\begin{thm}[Agler, \cite{Agler}]\label{thm_Agler}
An operator $T$ having the closed annulus $\overline{\A}_r$ as a spectral set, where
\[
\mathbb{A}_r=\{z \in \C : r<|z|<1\} \qquad (0<r<1),
\]  
dilates to a normal operator $N$ with spectrum $\sigma(N)$ lying on the boundary of $\mathbb A_r$.
\end{thm}

Apart from Agler's work \cite{Agler}, there is a vast literature on operator and function theoretic aspects of an annulus, e.g., see \cite{SarasonII, SarasonI, McCullough, Dritschel, Tsikalas, TsikalasII, TsikalasIII, Bello, Pas-McCull, PalI, PalII} and the references therein.

\smallskip

 In this article, we continue further the study of an operator having the closed annulus $\overline{\A}_r$ as a spectral set. Such an operator is called an $\A_r$-\textit{contraction}. Unitaries, isometries and pure isometries are special classes of contractions. There are analogous classes in the literature for $\A_r$-contractions. A normal operator that has its spectrum on the boundary $\partial \A_r$ is called an $\A_r$-\textit{unitary}. An $\A_r$-\textit{isometry} is an operator that has its spectrum inside $\CA_r$ and can be extended to an $\A_r$-unitary. A \textit{pure $\A_r$-isometry} is an $\A_r$-isometry whose restriction to any nonzero reducing subspace is not an $\A_r$-unitary. Thus, Agler's famous result (i.e., Theorem \ref{thm_Agler}) states that every $\A_r$-contraction dilates to an $\A_r$-unitary. An alternative proof to this theorem was given by Dritschel, Jury and McCullough in \cite{Dritschel} for the following special case.
	
\begin{thm}[\cite{Dritschel}, Theorem 6.5]\label{Annulus dilationII}
		An $\A_r$-contraction $T$ with $\sigma(T) \subseteq \A_r$ admits a normal $ b\CA_r$-dilation.		
	\end{thm}
Evidently, there is a gap between Theorem \ref{thm_Agler} and Theorem \ref{Annulus dilationII}. Indeed, Theorem \ref{thm_Agler} dilates every $\A_r$-contraction to an $\A_r$-unitary, whereas Theorem \ref{Annulus dilationII} guarantees $\A_r$-unitary dilation for an $\A_r$-contraction that has its spectrum inside the open annulus $\A_r$. In this article, we fill up this gap by an application of Theorem \ref{Annulus dilationII} and give a new proof to Theorem \ref{thm_Agler} in Section \ref{sec_Agler}. This is the first main result of this article.	
	 
	 \smallskip 
	
	In Section \ref{sec_var}, we consider the complex algebraic variety generated by the polynomial
	$
	q(z_1, z_2)=z_1z_2-\frac{r}{1+r^2}
	$
	and its intersection with the closure of the Euclidean biball $\mathbb{B}_2 = \{(z_1, z_2) \in \C^2 : |z_1|^2 + |z_2|^2 < 1\}$. We denote by $Z(q)$ the set of zeros of $q(z_1,z_2)$ and call the set $Z(q) \cap\overline{\B}_2$ the \textit{principal variety}. In the same section we prove that an invertible operator $T$ is an $\A_r$-contraction if and only if $Z(q) \cap\overline{\B}_2$ is a spectral set for the pair of operators
	\begin{equation}\label{eqn_k}
	\kappa(T)=\left(\frac{T}{\sqrt{1+r^2}}, \frac{rT^{-1}}{\sqrt{1+r^2}} \right).
	\end{equation}
The principal variety $Z(q) \cap\overline{\B}_2$ being a spectral set is easier to deal with than an object like $\overline{\A}_r$. The reason is that $Z(q) \cap\overline{\B}_2$ is polynomially convex but $\overline{\A}_r$ is not and Oka-Weil theorem (see CH-7 of \cite{Oka-Weil}) provides the freedom of testing von Neumann's inequality (as in (\ref{eqn:new-001}) \& (\ref{eqn:new-002})) in terms of polynomials only without stepping into the rational algebra for $Z(q) \cap\overline{\B}_2$, which is not possible in case of $\overline{\A}_r$. Note that the set $Z(q)\cap \overline{\mathbb B}_2$ appeared in \cite{Tsikalas} with a different purpose. Here we consider $Z(q)\cap \overline{\mathbb B}_2$ as a spectral set. In Section \ref{sec_var}, we show interplay between $\A_r$-contractions and the induced operator pairs $\kappa(T)$ using operator theory of the biball.

\smallskip

There are several wider classes of operators associated with an annulus other than the $\A_r$-contractions. For example, Bello and Yakubovich \cite{Bello} studied the following two classes of operators each of which is bigger than the class of $\A_r$-contractions:
	\[
	C_\alpha=\{ T: T \ \text{is invertible and} \ \alpha(T^*, T) \geq 0\} \ \ \& \ \ C_{1,r}=\{ T: T \ \text{is invertible and} \ \|T\|, \|rT^{-1}\|\leq 1\},
	\] 
	where $\alpha(T^*, T):=-T^{*2}T^2+(1+r^2)T^*T-r^2I$. On the other hand, McCullough and Pascoe \cite{Pas-McCull} explored important classes such as quantum annulus. The following theorem was proved in \cite{Bello}.
	
	\begin{thm} \label{thm_Ar_Ca}
	Every $\A_r$-contraction is in $C_\alpha$ and $C_\alpha \subsetneq C_{1,r}$.
	\end{thm} 

In Section \ref{sec_spher}, we give an alternative characterization of operators in $C_\alpha$ via spherical contractions. By a spherical contraction, we mean a commuting tuple of operators $(T_1, \dotsc, T_n)$ that satisfies $T_1^*T_1+\dotsc +T_n^*T_n \leq I$. More precisely, we show that an invertible operator $T \in C_\alpha$ if and only if the induced pair $\kappa(T)$ is a spherical contraction. In the same section, we feature the $C_{\alpha}$ class in light of spherical contractions.

\smallskip 

Bello and Yakubovich \cite{Bello} proved that operators in $C_\alpha$ and $C_{1,r}$ classes need not have $\CA_r$ as a spectral set. Also, it was shown in \cite{Bello} that $\CA_r$ is complete $\sqrt{2}$-spectral set for all operators in $C_\alpha$. Further, Tsikalas \cite{Tsikalas} proved that the bound $K=\sqrt{2}$ is the sharpest. Again, it was also proved in \cite{TsikalasII} that if $K$ is the smallest constant such that $\CA_r$ is a $K$-spectral set for all operators in $C_{1, r}$, then $K \geq 2$. However, the bound $2$ is optimal or not for $C_{1,r}$ remains an open problem. In this paper, we ask a slightly different question: what are the minimal spectral sets for all operators in $C_{\alpha}$ and $C_{1,r}$? Evidently, we are in search of a minimal compact set $X$ that contains $\CA_r$ as a proper subset and any operator in $C_\alpha$ or $C_{1,r}$ having its spectrum in $X$ satisfies von Neumann's inequality (\ref{eqn:new-001}) when $K=1$. So, here we hold $K=1$ in (\ref{eqn:new-001}) and expand the underlying compact set to something bigger than $\CA_r$ to have von Neumann's inequality for operators in $C_{\alpha}$ or $C_{1,r}$. Needless to mention, $\DC$ is a spectral set for $C_\alpha$ and $C_{1,r}$ as these classes consist of contractions. In Section \ref{sec_vNset}, we prove that $\DC$ is a minimal spectral set for $C_\alpha$ and $C_{1,r}$. En route, we consider the following classes:
\begin{equation}\label{eqn_102}
	C_\alpha^*=\{T: T^* \in C_\alpha\},  \ 	C_1=\{T : T \ \text{is an invertible contraction} \}.
\end{equation}
In the same section, we show that $\DC$ is a minimal spectral set for $C_{\alpha}^*, \,C_\alpha \cup C_{\alpha}^*$ and $C_1$. The map $\kappa : \mathbb{C} \setminus \{0\} \rightarrow \C^2$ defined by $\kappa(z)=\left( \frac{z}{\sqrt{1+r^2}}, \frac{rz^{-1}}{\sqrt{1+r^2}} \right)$ is holomorphic and first appeared in \cite{A-R-SW} with some scaling of its components. Then it was further studied in detail in \cite{Tsikalas, TsikalasII}. Since $\kappa$ is holomorphic, one naturally investigates the minimality of spectral sets for the classes $\kappa(C_\alpha), \kappa(C_{\alpha}^*), \kappa(C_\alpha \cup C_\alpha^*), \kappa(C_{1,r})$ and $\kappa(C_1)$, where for a class $\mathcal{C}$ of invertible operators, $\kappa(\mathcal{C}) = \{\kappa(T) : T \in \mathcal{C}\}$ and $\kappa(T)$ is as in \eqref{eqn_k}. However, it turns out that $\kappa(\DC \setminus \{0\})$ is an unbounded set and consequently the notion of spectral set does not work any further. To address this issue, we consider an analogue of spectral sets by removing the compactness condition.
\begin{defn}\label{defn:new-001}
A set $Y \subset \mathbb C^n$ is called a \textit{von Neumann set} for a commuting $n$-tuple of operators $\underline T =(T_1, \dots, T_n)$ acting on a Hilbert space if the Taylor joint spectrum $\sigma_T(\underline T) \subseteq Y$ and von Neumann's inequality (\ref{eqn:new-001}) with $K=1$ holds for all rational functions $f(z_1, \dots , z_n)$ that have no singularities in $Y$. Moreover, $Y$ is called a \textit{minimal (closed) von Neumann set} for a class of commuting $n$-tuples of Hilbert space operators $\mathcal S$, if $Y$ is a (closed) von Neumann set for each member of $\mathcal S$ and for any proper (closed) subset $Y_1$ of $Y$, there is $\underline S=(S_1, \dots , S_n) \in \mathcal S$ such that $Y_1$ is not a (closed) von Neumann set for $\underline S$.
\end{defn}
The first consequence of the notion of von Neumann set is Lemma \ref{lem_invc_vN}, which states that $T$ is an invertible contraction if and only if $\overline{\mathbb D} \setminus \{0\}$ is a von Neumann set for $T$. Next, we have Theorem \ref{thm_vN_iff} which builds a bridge between von Neumann set for an invertible operator $T$ and that of $\kappa(T)$ by showing that a set $Y \subseteq \mathbb C \setminus \{0\}$ is a von Neumann set for $T$ if and only if $\kappa(Y)$ is a von Neumann set for $\kappa(T)$. In the same section, we further prove that $\kappa(\overline{\D} \setminus \{0\})$ is a minimal closed von Neumann set for each of the following classes: $\kappa(C_\alpha), \ \kappa(C_\alpha^*), \ \kappa(C_\alpha \cup C_\alpha^*), \ \kappa(C_{1,r}), \ \kappa(C_1)$.

\smallskip

McCullough and Pascoe \cite{Pas-McCull} considered the annulus $A_r=\{z \in \C : r<|z|<r^{-1}\}$ for $r \in (0,1)$ and introduced the following classes of operators associated with it:
\begin{align*}
SA_r&=\{T: \overline{A}_r  \ \text{is a spectral set for} \ T \};\\
PA_r&=\{T\,:\, T \text{ is an invertible operator and } \, r^2+r^{-2}-T^*T-T^{-*}T^{-1} \geq 0\};\\
QA_r &= \{T\,:\, T \text{ is an invertible operator and } \, \|T\|, \, \|T^{-1}\| \leq r^{-1}  \}.
\end{align*}
The class $QA_r$ is referred to as the \textit{quantum annulus}. It was shown in \cite{PalI} by the authors that the map $\varphi: A_r \to \mathbb{A}_{r^2}$ given by $\varphi(z)=rz$ is a biholomorphism. By an application this map, we establish the equivalence of the class of $\A_r$-contractions, the $C_{\alpha}$ class and the $C_{1,r}$ class with $SA_r$, $PA_r$ and $QA_r$, respectively. In Section \ref{sec_quant}, we prove results for $SA_r$, $PA_r$, $QA_r$ that are analogous to the results of Section \ref{sec_vNset}. 

\section{Rational dilation on an annulus: an alternative approach}\label{sec_Agler}
	

\noindent Ever since von Neumann \cite{v-N} introduced the notion of spectral set and Sz. Nagy \cite{B-Nagy} discovered unitary dilation of a contraction, mathematicians used to study these fundamental concepts (i.e., spectral set, complete spectral set, rational dilation) independently for compact sets and tuples of commuting operators associated with them. However, it trivially follows that if a compact set $X\subset \C^n$ is a complete spectral set for a commuting $n$-tuple of operators $\underline{T}$, then it is a spectral set for $\underline{T}$. Later, Arveson \cite{ArvesonII} removed much of the mystery by connecting the two dots, namely the rational dilation and complete spectral set.

\begin{thm}[Arveson, \cite{ArvesonII}]\label{thm_Arveson} Let $\underline{T}$ be a tuple of commuting operators on a Hilbert space $\HS$ and let $X \subset \C^n$ be a compact set. Then $X$ is a complete spectral set for $\underline{T}$ if and only if $\underline{T}$ has a normal $bX$-dilation.	
\end{thm}

 In this section, we shall prove Agler's theorem (i.e., Theorem \ref{thm_Agler}) independently by repeated applications of Theorem \ref{thm_Arveson}. However, our starting point will be Theorem \ref{Annulus dilationII} due to Dritschel-Jury-McCullough \cite{Dritschel} and McCullough \cite{McCullough}. We begin with the following elementary result which will be useful in this context.

\begin{lem}\label{thm_nested}
Let $\{X_n\}_{n \in \N}$ be a sequence of non-empty compact sets in $\C$ such that $X_n \supseteq X_{n+1}$ for every $n \in \N$. Then $X=\underset{n \in \N}{\bigcap}X_n$ is a (complete) spectral set for $T$ if and only if each $X_n$ is a (complete) spectral set for $T$.  
\end{lem}
	
\begin{proof}

It follows from the definition of (complete) spectral set that a compact superset of a (complete) spectral set is always a (complete) spectral set. So, the forward part is obvious.

\smallskip

We prove the converse part for complete spectral set only as the proof for the spectral set is similar. Conversely, suppose each $X_n$ is a complete spectral set for $T$. Let $n \in \N$. Then $\sigma(T) \subseteq X_n$ and for any matricial rational function $[f_{ij}]_{k \times k} \in M_k(Rat(X_n))$, we have 
\begin{equation}\label{eqn_nested}
\|[f_{ij}(T)]_{k \times k}\| \leq \sup \left\{\|[f_{ij}(x)]_{k \times k}\| : x \in X_n \right\}.
\end{equation}
It is evident that $\sigma(T) \subseteq X$. Let $f \in Rat(X)$. Then one can find a large enough positive integer $N$ such that $f$ has poles outside $X_N$. If not, then each $X_n$ must contain a pole, say $x_n$, of $f$. Since $f$ has only finitely many poles, there is a pole of $f$, say $\alpha$, that occurs infinitely many times in the sequence $\{x_n\}$. Thus, $\alpha$ must belong to infinitely many members of the family of compact sets $\{X_n\}$. If $\alpha \notin X$, then there exists some $m\in \N$ such that $\alpha \notin X_m$ and so, $\alpha \notin X_n$ for all $n \geq m$. This shows that $\alpha$ belongs to only finitely many $X_n$'s which is not true. So, we have that $\alpha \in X$ which gives a contradiction to the fact that $f \in Rat(X)$. Hence, there exists $N \in \N$ such that $f \in Rat(X_N)$ and so, $f \in Rat(X_n)$ for $n \geq N$.

\smallskip 

Let $[f_{ij}]_{k \times k} \in M_k(Rat(X))$ for $k \in \N$. Since there are only finitely any $f_{ij}$'s, one can choose $n_0$ large enough so that each $f_{ij}$  has poles outside $X_{n_0}$. Thus, $[f_{ij}]_{k \times k} \in M_k(Rat(X_{n}))$ for all $n \geq n_0$. Let $n \geq n_0$. Since $X_n$ is compact, there is $x_n \in X_n$ such that
\begin{equation}\label{eqn_502}
\|[f_{ij}(T)]_{k \times k}\| \leq \sup \left\{\|[f_{ij}(x)]_{k \times k}\| : x \in X_n \right\}=\|[f_{ij}(x_n)]_{k \times k}\|.
\end{equation}
Since $X_n \supseteq X_{n+1}$ for all $n \in \N$, the sequence $\{x_n\}_{n \geq n_0} \subseteq X_{n_0}$ is bounded. So, $\{x_n\}_{n \geq n_0}$ has a convergent subsequence, say $\{x_{n_t}\}$  and let its limit be $x_0$. Clearly, each $x_{n_t} \in X_{n_1}$ and so, $x_0 \in X_{n_1}$. Similarly, $x_{n_t} \in X_{n_2}$ for $t \geq 2$ and so, $x_0 \in X_{n_2}$. Continuing this way, we have that $x_0 \in  X_{n_t}$ for all $t \in \N$ and so, $x_0$ belongs to infinitely many $X_n$'s. As discussed above, this is possible only if $x_0 \in X$. Letting $n_t \to \infty$ in (\ref{eqn_502}), we have that 
\[
\|[f_{ij}(T)]_{k \times k}\| \leq \|[f_{ij}(x_0)]_{k \times k}\| \leq \sup \left\{\|[f_{ij}(x)]_{k \times k}\| : x \in X \right\}.
\]
Hence, $X$ is a complete spectral set for $T$ and the proof is complete. 
\end{proof}

\noindent \textbf{\textit{Proof of Theorem \ref{thm_Agler}}.} Suppose $T$ is an $\A_r$-contraction. We intend to show that $T$ admits a normal $\partial {\A}_r$-dilation. By Theorem \ref{thm_Arveson}, it suffices if we show that $\CA_r$ is a complete spectral set for $T$. Let $ k \in \N$ be arbitrary and consider any matricial rational function $[f_{ij}]_{k \times k} \in M_k(Rat(\CA_r))$. One can choose sufficiently small $\epsilon \in (0, r)$ such that every $f_{ij} \ ( 1 \leq i, j \leq k)$ is holomorphic on the closed annulus $\overline{\A}_{r-\epsilon}^{1+\epsilon}$, where $\A_{r-\epsilon}^{1+\epsilon}=\{z \in\mathbb{C}: r-\epsilon <|z| < 1+\epsilon \}$. Needless to mention that the choice of $\epsilon$ depends on $f_{ij}$ for $ 1 \leq i, j \leq k$. Note that $\overline{\A}_r \subsetneq \A_{r-\epsilon}^{1+\epsilon}$. Set $s=\frac{r-\epsilon}{1+\epsilon}$. The map $\phi: \overline{\A}_{r-\epsilon}^{1+\epsilon} \to \overline{\A}_s$ defined by $\displaystyle \phi(z)=\frac{z}{1+\epsilon}$ is a biholomorphism. Let us define $J=\phi(T)$. It follows from spectral mapping theorem that
\begin{equation*}
\begin{split}
\sigma(J)=\sigma(\phi(T))=\phi(\sigma(T)) \subseteq \phi(\overline{\A}_r) \subseteq \phi(\A_{r-\epsilon}^{1+\epsilon}) \subseteq \A_s\,.
\end{split}
\end{equation*}
Let $g \in Rat(\CA_s)$ be arbitrary. Then  
$g\circ \phi \in Rat(\CA_r)$ and since $T$ is an $\A_r$-contraction, we have that
	\begin{equation*}
		\begin{split}
			\|g(J)\|=\|(g\circ \phi)(T)\| 
			 \leq \sup\{|g\circ \phi(z)|: z \in \overline{\A}_r\} 
			 	 \leq  \|g\|_{\infty, \CA_s}.\\
		\end{split}
	\end{equation*}
	Consequently, $J$ is an $\A_s$-contraction with $\sigma(J) \subseteq \A_s$. It follows from Theorem \ref{Annulus dilationII} that $J$ admits a normal $\partial \A_s$-dilation and so, by Theorem \ref{thm_Arveson}, $\overline{\A}_s$ is a complete spectral set for $J$. Therefore, 
	\begin{equation}\label{eqn_503}
		\begin{split}
			\|[f_{ij}(T)]_{k \times k}\|
			=\|[(f_{ij}\circ \phi^{-1})(J)]_{k \times k}\|
			 \leq \sup \left\{\|[f_{ij}\circ \phi^{-1}(z)]_{k \times k}\| : z \in \overline{\A}_s \right\},
		\end{split}
	\end{equation}
	where the last inequality follows because $[f_{ij}\circ \phi^{-1}]_{k \times k} \in M_k(Rat(\CA_s))$ and $\CA_s$ is a complete spectral set for $J$. Using (\ref{eqn_503}), we have
	\begin{equation}\label{eqn:4.1}
		\begin{split}
			\|[f_{ij}(T)]_{k \times k}\| 
			 \leq \sup \left\{\|[f_{ij}(w)]_{k \times k}\| : w \in \overline{\A}_{r-\epsilon}^{1+\epsilon} \right\}.\\
		\end{split}
	\end{equation}
	The above inequality holds for all sufficiently small $\epsilon>0$. For sufficiently large  $n_0 \in \mathbb{N}$, it follows from (\ref{eqn:4.1}) that $\|[f_{ij}(T)]_{k \times k}\|	 \leq \sup \left\{\|[f_{ij}(x)]_{k \times k}\| : x \in X_n\right\}$ for all $n \geq n_0$, where $X_n=\CA_{r-1 \slash n}^{1+1\slash n}$. Note that $X_n \supseteq X_{n+1}$ and $\bigcap_{n \geq n_0}X_n=\CA_r$. Following the proof of Theorem \ref{thm_nested}, one can show that $\CA_r$ is a complete spectral set for $T$. An application of Theorem \ref{thm_Arveson} now concludes the proof.
\qed

\section{The $\A_r$-contractions and the principal variety in the biball}\label{sec_var}

\noindent In this section, we present an alternative characterization for $\A_r$-contractions via a certain variety intersecting the biball $\mathbb{B}_2$, where $\mathbb{B}_2=\{(z_1, z_2) \in \C^2 : |z_1|^2+|z_2|^2 <1 \}$. As mentioned before, this variety is set of zeros of the following polynomial:
\begin{equation} \label{eqn:new-T-001}
q(z_1, z_2)=z_1z_2-\frac{r}{1+r^2}.
\end{equation}
The set $Z(q) \cap \overline{\B}_2$ is referred to as the principal variety and plays a crucial role in the study of $\A_r$-contractions. Unlike the closed annulus $\CA_r$, the principal variety is polynomially convex. This allows to study an $\A_r$-contraction $T$ in a simpler way in terms of the induced pair $\kappa(T)$ (as in \ref{eqn_k}) with the help of polynomial approximation on $Z(q) \cap \overline{\B}_2$ in view of Oka-Weil theorem (see CH-7 of \cite{Oka-Weil}). However, before displaying our results let us recall a few useful terminologies and results related to the operators associated with the biball or even more general polyball and their interplay with spherical contractions. 

\begin{defn} \label{defn:new-T-01}
A tuple of commuting operators $(T_1,\dots , T_n)$ for which the closed polyball $\overline{\mathbb B}_n=\{(z_1, \dots , z_n)\in \C^n \,:\, |z_1|^2+\dots +|z_n|^2 \leq 1  \}$ is a spectral set is called a \textit{$\mathbb B_n$-contraction}. A $\mathbb B_n$-\textit{unitary} is a tuple of commuting normal operators $(T_1,\dots, T_n)$ with its Taylor joint spectrum $\sigma_T(T_1,\dots,T_n) \subseteq b\overline{\mathbb B}_n$, where $b\overline{\mathbb B}_n$ is the distinguished boundary of $\overline{\mathbb B}_n$. A $\mathbb B_n$-\textit{isometry} is the restriction of a $\mathbb B_n$-unitary $(T_1,\dots, T_n)$ to a joint invariant subspace of $T_1, \dots, T_n$. Also, a \textit{completely non-unitary $\B_n$-contraction} or simply a \textit{c.n.u. $\B_n$-contraction} is a $\B_n$-contraction whose restriction to any nonzero joint reducing subspace of its components is not a $\B_n$-unitary.
\end{defn}
\begin{defn}\label{eqn:new-T-02}
A commuting tuple of operators $\UT=(T_1, \dotsc, T_n)$ is said to be
		\begin{enumerate}
			\item a \textit{spherical contraction} if $T_1^*T_1 + \dotsc + T_n^*T_n \leq I$;			
			\item a \textit{spherical unitary} if each $T_j$ is normal and $T_1^*T_1 + \dotsc + T_n^*T_n = I$;
			
			\item a \textit{spherical isometry} if $T_1^*T_1 + \dotsc + T_n^*T_n =I$.
		\end{enumerate}
\end{defn}
Not every spherical contraction is a $\B_n$-contraction, see Arveson's work \cite{ArvesonIII} for further details. However, the next result shows that these classes coincide at the level of unitaries and isometries.	
	\begin{thm}[\cite{AthavaleIII}, Proposition 2]\label{thm_Bn_iso}
		Let $\underline{T}=(T_1, \dotsc, T_n)$ be a tuple of commuting operators acting on a Hilbert space $\mathcal{H}$. Then $\underline{T}$ is a $\B_n$-unitary (respectively, $\B_n$-isometry) if and only if $\underline{T}$ is a spherical unitary (respectively, spherical isometry).
	\end{thm}
	
The canonical decomposition of a contraction (see CH-I of \cite{Nagy}) splits a contraction $T$ into two orthogonal parts of which one is a unitary and the other is a completely non-unitary (c.n.u.) contraction. Eschmeier \cite{EschmeierII} found analogous orthogonal decomposition for a $\B_n$-contraction.

\begin{thm}[Eschmeier, \cite{EschmeierII}]\label{thm_decomp-Ball}
		Let $\underline{T}=(T_1, \dotsc, T_n)$ be a $\B_n$-contraction on a Hilbert space $\mathcal{H}$. Then there is an orthogonal decomposition of $\mathcal{H}$ into joint reducing subspaces $\mathcal{H}_u$ and $\mathcal{H}_c$ of $\underline T$ such that 
		 $(T_1|_{\mathcal{H}_u}, \dotsc, T_n|_{\mathcal{H}_u})$ is a $\B_n$-unitary and $(T_1|_{\mathcal{H}_c}, \dotsc, T_n|_{\mathcal{H}_c})$ is a c.n.u. $\B_n$-contraction.
	\end{thm}
	
Below, we recall a few results on $\A_r$-contractions that will be used in sequel.	
	\begin{prop}[\cite{PalII}, Proposition 2.4]\label{prop_A_unit} A normal operator $T$ acting on a Hilbert space $\mathcal{H}$ is an $\A_r$-unitary if and only if $(I-T^*T)(T^*T-r^2I)=0$.
\end{prop}

Next, we provide an algebraic characterization for $\A_r$-isometries.

\begin{thm}[\cite{Bello}, Theorem 2.8 \& \cite{PalII}, Theorem 3.6]\label{thm_A_iso}
Let $V$ be an invertible operator on a Hilbert space $\HS$. Then the following are equivalent:
\begin{enumerate}
\item[(i)] $V$ is an $\A_r$-isometry;
\item[(ii)] $-V^{*2}V^2+(1+r^2)V^*V-r^2I=0$, or equivalently $V^*V+(rV^{-1})^*rV^{-1}=(1+r^2)I$.
\end{enumerate}
\end{thm}

Now we are in a position to begin our investigation on relation between $\A_r$-contractions and $\B_2$-contractions. We start with a lemma that sets a connection between the domains $\A_r$ and $\B_2$. 

\begin{lem}\label{lem_pi}
The map given by
 \[
 \kappa: \C\setminus \{0\} \to \C^2, \quad \kappa(z)=\left(\frac{z}{\sqrt{1+r^2}}, \frac{rz^{-1}}{\sqrt{1+r^2}} \right)
 \]
is a holomorphic closed map such that
\begin{align*}
(i)\ \kappa(\CA_r)=Z(q) \cap \overline{\B}_2 \quad (ii) \ \kappa(\A_r)= Z(q) \cap \B_2 \quad (iii)\ \kappa(\partial \CA_r)=Z(q) \cap \partial \B_2.
\end{align*}
\end{lem}

\begin{proof}
It is evident that $\kappa$ is a holomorphic map. Let $Y$ be a closed subset of $\C \setminus \{0\}$. Then $Y=C \cap (\C \setminus \{0\})=C \setminus \{0\}$ for some closed subset $C$ in $\C$. We show that $\kappa(Y)$ is a closed subset of $\C^2$. Take a sequence $\{y_n\}$ in $Y$ such that
\[
\kappa(y_n)=\left(\frac{y_n}{\sqrt{1+r^2}}, \frac{ry_n^{-1}}{\sqrt{1+r^2}} \right)
\]
converges to $(\alpha_1, \alpha_2)$ as $n \to \infty$. Then 
\[
\alpha_1\alpha_2=\lim_{n \to \infty}\frac{y_n}{\sqrt{1+r^2}}\frac{ry_n^{-1}}{\sqrt{1+r^2}}=\frac{r}{1+r^2} \quad \text{and so,} \quad \kappa\left(\alpha_1\sqrt{1+r^2}\right)=\left(\alpha_1, \frac{r}{(1+r^2)\alpha_1}\right)=(\alpha_1, \alpha_2).
\] 
Since $Y$ is closed and $\{y_n\} \subset Y$ converges to $\alpha_1\sqrt{1+r^2}$, we have that $(\alpha_1, \alpha_2) \in \kappa(Y)$ and thus, $\kappa(Y)$ is a closed subset in $\C^2$. Consequently, $\kappa$ is a closed map. It is easy to see that
\begin{equation*}
\begin{split}
\left(\frac{z}{\sqrt{1+r^2}}, \frac{rz^{-1}}{\sqrt{1+r^2}} \right) \in \overline{\B}_2 \iff |z|^2+r^2|z|^{-2} \leq 1+r^2
 & \iff |z|^4-(1+r^2)|z|^2+r^2 \leq 0\\
 &\iff(|z|^2-r^2)(|z|^2-1) \leq 0\\
 & \iff z \in \CA_r.
\end{split}
\end{equation*}
Thus, $\kappa(\CA_r)=Z(q) \cap \overline{\B}_2$ and the other parts of the theorem follow similarly.
\end{proof}

In view of Lemma \ref{lem_pi}, it is natural to ask if the above result can be lifted to the level of operators, which is to say that for an $\A_r$-contraction $T$ if $Z(q) \cap \BC$ is a spectral set for $\kappa(T)$ and vice-versa. We show here that the answer to this is affirmative. Let us start with a proposition.

\begin{prop}\label{prop_A_uni_Ball}
Let $U$ be an invertible operator acting on a Hilbert space $\HS$ and let $q$ be the polynomial as in $(\ref{eqn:new-T-001})$. Then the following are equivalent:
\begin{enumerate}
\item $U$ is an $\A_r$-unitary; 
\item $\kappa(U)$ is a $\B_2$-unitary;
\item $U$ is normal and $Z(q) \cap \partial \B_2$ is a spectral set for $\kappa(U)$.  
\end{enumerate}

\end{prop}

\begin{proof}
We shall prove $(1) \iff (2) \iff (3)$.

\smallskip

\noindent $(1) \iff (2)$. We have by Theorem \ref{thm_Bn_iso} and Proposition \ref{prop_A_unit} that
\begin{equation*}
\begin{split}
\text{ $U$ is an $\A_r$-unitary}
& \iff \text{ $U$ is normal and} \ (I-U^*U)(U^*U-r^2I)=0 \\
& \iff \text{ $U$ is normal and} \ U^*U+r^2U^{-*}U^{-1}=(1+r^2)I\\
& \iff \kappa(U)=\left(\frac{U}{\sqrt{1+r^2}}, \frac{rU^{-1}}{\sqrt{1+r^2}} \right) \ \mbox{is a $\B_2$-unitary}.
\end{split}
\end{equation*}

\noindent $(2) \implies (3)$. Let $\kappa(U)$ be a $\B_2$-unitary. Then $\sigma_T(\kappa(U)) \subseteq \partial\B_2$ and by spectral mapping principle, $\sigma_T(\kappa(U))=\kappa(\sigma(U)) \subseteq Z(q)$. Thus, $\sigma_T(\kappa(U)) \subseteq Z(q) \cap \partial \B_2$. Since $\kappa(U)$ is a commuting pair of normal operators, we have that any compact set containing $\sigma_T(\kappa(U))$ is a spectral for $\kappa(U)$. Hence, $Z(q) \cap \partial \B_2$ is a spectral set for $\kappa(U)$.

\smallskip 

\noindent $(3) \implies (2)$. Suppose that $U$ is normal and $Z(q) \cap \partial \B_2$ is a spectral set for $\kappa(U)$. Consequently, $\kappa(U)$ is a commuting pair of normal operators and $\partial \B_2$ is a spectral set for $\kappa(U)$. The latter holds, because any compact superset containing a spectral set is also a spectral set. Hence, $\kappa(U)$ is a $\B_2$-unitary and this completes the proof. 
\end{proof}

A natural analogue of the preceding result holds for $\A_r$-isometry versus $\B_2$-isometry.

\begin{prop}\label{thm_A_iso_Ball}
Let $V$ be an invertible operator acting on a Hilbert space $\HS$. Then $V$ is an $\A_r$-isometry if and only if $\kappa(V)$ is a $\B_2$-isometry.
\end{prop} 

\begin{proof}
Let us denote
\[
(V_1, V_2)=\kappa(V)=\left(\frac{V}{\sqrt{1+r^2}}, \frac{rV^{-1}}{\sqrt{1+r^2}} \right).
\]
If $V$ is an $\A_r$-isometry, then it follows from Theorems \ref{thm_Bn_iso} and \ref{thm_A_iso} that $\kappa(V)$ is a $\B_2$-isometry. Conversely, let $\kappa(V)$ be a $\B_2$-isometry. Then there exist a space $\mathcal{K}$ containing $\HS$ and a $\B_2$-unitary $(N_1, N_2)$ on $\mathcal{K}$ such that $(N_1|_{\HS}, N_2|_{\HS})=(V_1, V_2)$. Define
\[
\mathcal{L}:=\overline{span}\{N_1^{*n}N_2^{*m}k : \ \ k\in \mathcal K, n , m \in \mathbb{N} \cup \{0\} \}.
\]
Evidently, $\mathcal L$ is a closed joint reducing subspace for $N_1$ and $N_2$. Set $(U_1, U_2)=(N_1|_{\mathcal{L}}, N_2|_{\mathcal{L}})$, which is the minimal normal extension of $(V_1, V_2)$. By Theorem  \ref{thm_Bn_iso}, $N_1^*N_1+N_2^*N_2=I$. Thus, $(U_1|_{\HS}, U_2|_{\HS})=(V_1, V_2)$ and $U_1^*U_1+U_2^*U_2=I$. For any $h \in \HS$, we have that 
\[
U_1U_2h=V_1V_2h=\frac{r}{1+r^2}h \quad \text{and so,} \quad U_1U_2(N_1^{*n}N_2^{*m}h)=N_1^{*n}N_2^{*m}U_1U_2h=\frac{r}{1+r^2}(N_1^{*n}N_2^{*m}h)
\]
for $m, n \in \mathbb{N} \cup \{0\}$. By continuity arguments, $U_1U_2=\frac{r}{1+r^2}I$. Define $U=\sqrt{1+r^2}U_1$ on $\mathcal{L}$. Consequently, we have that $V=U|_{\HS}$ and $U^*U+r^2U^{-*}U^{-1}=(1+r^2)I$. By Proposition \ref{prop_A_unit}, $U$ is an $\A_r$-unitary. If $z \in \sigma(V)$, then $\kappa(z) \in \kappa(\sigma(V))=\sigma_T(\kappa(V)) \subseteq \BC$. It follows from Lemma \ref{lem_pi} that $z \in \CA_r$ and so, $\sigma(V) \subseteq \CA_r$. Thus, $V$ is an $\A_r$-isometry which completes the proof.
\end{proof}

As an application of the above proposition, below we give an alternative proof to the Wold decomposition of an $\A_r$-isometry that appears as Theorem 3.13 in \cite{PalII}.

\begin{thm}\label{thm_Wold}
	Let $V$ be an $\A_r$-isometry on a Hilbert space $\mathcal{H}$. Then there exists a decomposition of $\mathcal{H}$ into an orthogonal sum of two closed subspaces reducing $V$, say $\mathcal{H}= \mathcal{H}_{1} \oplus \mathcal{H}_{2},$ such that $V|_{\mathcal{H}_{1}}$ is an $\A_r$-unitary and $V|_{\mathcal{H}_{2}}$ is a pure $\A_r$-isometry. The space $\mathcal{H}_{1}$ or $\mathcal{H}_{2}$ may equal the trivial subspace $\{0\}$.
\end{thm}

\begin{proof}
Let $(V_1, V_2)=\kappa(V)$. We have by Proposition \ref{thm_A_iso_Ball} that $\kappa(V)$ is a $\B_2$-isometry. It follows from Theorem \ref{thm_decomp-Ball} that there is an orthogonal decomposition of $\HS$ into joint reducing subspaces $\HS_1$ and $\HS_2$ of $(V_1, V_2)$ such that $(V_1|_{\HS_1}, V_2|_{\HS_1})$ is a $\B_2$-unitary and $(V_1|_{\HS_2}, V_2|_{\HS_2})$ is a c.n.u. $\B_2$-isometry. By Proposition \ref{prop_A_uni_Ball}, $V|_{\HS_1}$ is an $\A_r$-unitary. Let $\mathcal{L} \subseteq \HS_2$ be a closed reducing subspace of $V$ such that $V|_{\mathcal{L}}$ is an $\A_r$-unitary. Again by Proposition \ref{prop_A_uni_Ball}, $(V_1|_{\mathcal{L}}, V_2|_{\mathcal{L}})$ is a $\B_2$-unitary. Thus, $\mathcal{L}=\{0\}$ since $(V_1|_{\HS_2}, V_2|_{\HS_2})$ is a c.n.u. $\B_2$-isometry. Hence, $V|_{\HS_2}$ is a pure $\A_r$-isometry.
\end{proof}

We now prove a main result of this section and for this we need an elementary lemma whose proof is a routine exercise. However, we include a short proof here for the sake of completeness.

	\begin{lem}\label{lem_pconvex} 
		Let $K$ be a polynomially convex set in $\C^n$ and let $p \in \C[z_1, \dotsc, z_n]$ be such that $Z(p) \cap K \ne \emptyset$. Then $Z(p) \cap K$ is a polynomially convex set in $\C^n$. 
	\end{lem}
	
\begin{proof} Let $X=Z(p) \cap K$. To prove $X$ is polynomially convex, it suffices to show that for any $x \notin X$, there exists a polynomial $f \in \C[z_1, \dotsc, z_n]$ such that $\|f\|_{\infty, X}<|f(x)|$. Take any $x \notin X$. We discuss the two cases here depending on whether $x$ lies in $K$ or not. Let $x \notin K$. Then there exists $f \in \C[z_1, \dotsc, z_n]$ such that $\|f\|_{\infty, K} <|f(x)|$ since $K$ is polynomially convex. Thus $\|f\|_{\infty, X} \leq \|f\|_{\infty, K} <|f(x)|$. For $x \in K$, we have $x \notin Z(p)$ and so, $|p(x)|>0=\|p\|_{\infty, X}$ and we are done.	\end{proof}

\begin{thm}\label{thm_A_con_Ball}
Let $T$ be an invertible operator acting on a Hilbert space $\HS$. Then $T$ is an $\A_r$-contraction if and only if $Z(q) \cap \BC$ is a spectral set for $\kappa(T)$, where $q$ is as in $(\ref{eqn:new-T-001})$.
\end{thm}

\begin{proof}
Set  
\[
(T_1, T_2)=\kappa(T)=\left(\frac{T}{\sqrt{1+r^2}}, \frac{rT^{-1}}{\sqrt{1+r^2}} \right).
\]
Let $T$ be an $\A_r$-contraction and let $(z_1, z_2) \in \sigma_T(T_1, T_2)$. By spectral mapping theorem and Lemma \ref{lem_pi}, $(z_1, z_2) \in \sigma_T(\kappa(T))=\kappa(\sigma(T)) \subseteq \kappa(\CA_r)=Z(q) \cap \BC$. Thus, $\sigma_T(T_1, T_2) \subseteq Z(q) \cap \BC$. Now $\BC$ is convex and hence, is polynomially convex. So, it follows from Lemma \ref{lem_pconvex} that $Z(q) \cap \BC$ is polynomially convex. By Oka-Weil theorem (see CH-7 of \cite{Oka-Weil}), every holomorphic function on $Z(q) \cap \BC$ can be uniformly approximated by   a sequence of polynomials over $Z(q) \cap \BC$. Thus, it suffices to prove the von Neumann's inequality only for the polynomials. Let $f \in \C[z_1, z_2]$. Then, the fact that $T$ is an $\A_r$-contraction along with Lemma \ref{lem_pi} implies that 
\begin{equation*}
\begin{split}
\|f(T_1, T_2)\|  =\|f \circ \kappa(T) \| 
& \leq \sup \{|f\circ \kappa(z)| : z \in \CA_r \} \\
& \leq \sup\{|f(z_1, z_2)| : (z_1, z_2) \in Z(q) \cap \BC\}
\end{split}
\end{equation*}
and so, $(T_1, T_2)$ has the principal variety $Z(q) \cap \BC$ as a spectral set. Conversely, assume that $Z(q) \cap \BC$ is a spectral set for $\kappa(T)=(T_1, T_2)$. Let $z \in \sigma(T)$. Again, by spectral mapping theorem, we have that $\kappa(z) \in \kappa(\sigma(T))= \sigma_T(\kappa(T))\subseteq Z(q) \cap \BC$. By Lemma \ref{lem_pi}, $z \in \CA_r$ and so, $\sigma(T) \subseteq \CA_r$. Let $f$ be a Laurent polynomial, i.e., ${\displaystyle f(z)=\sum_{i=-n}^{m}a_iz^i }$, where each $a_i \in \C$ and $n, m \in \mathbb{N}$. We can re-write $f$ as
\[
f(z)=f_1\left(\frac{z}{\sqrt{1+r^2}}\right)+f_2\left(\frac{rz^{-1}}{\sqrt{1+r^2}}\right)
\]
for some polynomials $f_1$ and $f_2$ in one-variable. For the holomorphic map $g: \BC \to \C$ defined as $g(z_1, z_2)=f_1(z_1)+f_2(z_2)$, we have that
\begin{equation*}
\begin{split}
\|f(T)\|
&=\|f_1(T_1)+f_2(T_2)\| \\
&=\|g(T_1, T_2)\|\\
&\leq \sup\{|g(z_1, z_2)|: (z_1, z_2) \in Z(q) \cap \BC \} \qquad  [\text{as $Z(q) \cap \BC$ is a spectral set for $(T_1, T_2)$}]\\
&=\sup\{|f_1(z_1)+f_2(z_2)| : (z_1, z_2) \in Z(q) \cap \BC\}\\
&=\sup\left\{\left|f_1(z)+f_2\left(\frac{rz^{-1}}{1+r^2}\right)\right| : \left(z,\frac{rz^{-1}}{1+r^2}\right) \in  \BC\right\}\\
&=\sup\left\{\left|f(z\sqrt{1+r^2})\right| : \left(z,\frac{rz^{-1}}{1+r^2}\right) \in  \BC\right\}\\
&\leq \sup\left\{\left|f(z\sqrt{1+r^2})\right| : z\sqrt{1+r^2} \in \CA_r\right\} \quad [\text{by Lemma} \ \ref{lem_pi}]\\
& \leq \|f\|_{\infty, \CA_r}.
\end{split}
\end{equation*}
Since Laurent polynomials are dense in $Rat(\CA_r)$, it follows that $\|f(T)\| \leq \|f\|_{\infty, \CA_r}$ for every $f$ in $Rat(\CA_r)$. Consequently, $T$ is an $\A_r$-contraction. The proof is now complete.
\end{proof}

If $T$ is an $\A_r$-contraction, then it follows from Theorem \ref{thm_A_con_Ball} that $Z(q) \cap \BC$ is a spectral set for $\kappa(T)$. Since any compact superset containing a spectral set is again a spectral set, we have that $\BC$ is a spectral set for $\kappa(T)$. Thus, we have the following result.

\begin{cor}\label{cor_A_con_Ball}
If $T$ is an $\A_r$-contraction, then $\kappa(T)$ is a $\B_2$-contraction. 
\end{cor}

As an immediate application of the above corollary, we have an alternative proof to the following canonical decomposition for an $\A_r$-contraction that appeared as Theorem 4.5 in \cite{PalII}. Before stating it, let us recall from \cite{PalII} that a \textit{completely non-unitary $\A_r$-contraction} or simply a \textit{c.n.u. $\A_r$-contraction} is an $\A_r$-contraction whose restriction to any nonzero joint reducing subspace of its components is not an $\A_r$-unitary.

\begin{thm}\label{thm_ann_dec}
Let $T$ be an $\A_r$-contraction on a Hilbert space $\mathcal{H}$. Then there is a decomposition of $\mathcal{H}$ into an orthogonal sum of two closed subspaces reducing $T$, say $\mathcal{H}= \mathcal{H}_{1} \oplus \mathcal{H}_{2},$ such that $T|_{\mathcal{H}_1}$ is an $\A_r$-unitary and $T|_{\mathcal{H}_2}$ is a c.n.u. $\A_r$-contraction. The spaces $\mathcal{H}_{1}$ or $\mathcal{H}_{2}$ may equal $\{0\}$.
\end{thm}

\begin{proof}
The proof is similar to that of Theorem \ref{thm_Wold}. Let $(T_1, T_2)=\kappa(T)$. By Corollary \ref{cor_A_con_Ball}, $\kappa(T)$ is a $\B_2$-contraction. We have by Theorem \ref{thm_decomp-Ball} that there is an orthogonal decomposition of $\HS$ into joint reducing subspaces $\HS_1$ and $\HS_2$ of $(T_1, T_2)$ such that $(T_1|_{\HS_1}, T_2|_{\HS_1})$ is a $\B_2$-unitary and $(T_1|_{\HS_2}, T_2|_{\HS_2})$ is a c.n.u. $\B_2$-contraction. It follows from Proposition \ref{prop_A_uni_Ball} that $T|_{\HS_1}$ is an $\A_r$-unitary. Let $\mathcal{L} \subseteq \HS_2$ be a closed reducing subspace of $T$ such that $T|_{\mathcal{L}}$ is an $\A_r$-unitary. Again by Proposition \ref{prop_A_uni_Ball}, $(T_1|_{\mathcal{L}}, T_2|_{\mathcal{L}})$ is a $\B_2$-unitary. Thus, $\mathcal{L}=\{0\}$ since $(T_1|_{\HS_2}, T_2|_{\HS_2})$ is a c.n.u. $\B_2$-contraction and so, $T|_{\HS_2}$ is a c.n.u. $\A_r$-contraction.
\end{proof}

Evidently, Theorem \ref{thm_A_con_Ball} paves a way to realize an $\A_r$-contraction as a commuting pair of operators whose spectral set is polynomially convex. This leads to studying the set $Z(q) \cap \BC$ in more detail. In this regard, we first need a few terminologies and results from \cite{Dales} related to peak points and the Shilov boundary of a function algebra.

\smallskip 

\noindent \textbf{Boundary and peak points for finitely generated function algebras.}  Let $X$ be a compact Hausdorff space and let $C(X)$ be the Banach algebra of complex-valued continuous functions on $X$ endowed with the supremum norm, $\|f\|_\infty=\sup\{|f(x)|: x \in X\}$. A subalgebra of $C(X)$ is said to be a \textit{normed function algebra} if it contains the constants and separates the points of $X$. Let $\mathcal{A}$ be a normed function algebra on $X$. If $\mathcal{F}$ is a subset of $\mathcal{A}$, then the smallest closed subalgebra $\mathcal{A}_0$ of $\mathcal{A}$ which contains $\mathcal{F}$ and the identity of $\mathcal{A}$ is called the subalgebra of $\mathcal{A}$ generated by $\mathcal{F}$, and $\mathcal{F}$ is the \textit{set of generators} for $\mathcal{A}_0$. We say $\mathcal{A}$ is \textit{finitely generated} if it has a finite set of generators. Moreover, $\mathcal{A}$ is said to be \textit{inverse-closed} if every function of the algebra which has no zeros on $X$ has an inverse in $\mathcal{A}$. The \textit{Shilov boundary for $\mathcal{A}$} is the smallest closed subset of $X$ on which every function in $\mathcal{A}$ attains its maximum modulus. A point $w \in X$ is a \textit{peak point} if there exists $f \in \mathcal A$ such that $f(w)=1$ and $|f(z)|<1$ for all $z \ne w$ in $X$. The function $f$ is called a \textit{peaking function} for $w$. See Section 2 in \cite{Dales} for a rigorous discussion on this topic. Also, recall that the distinguished boundary $bX$ of $X$ is the Shilov boundary for $Rat(X)$, which may or may not be equal to the topological boundary $\partial X$ of $X$. For example, in case of the polyball $\B_n$, $\partial \overline{\B}_n=b\overline{\B}_n$ (see Example 4.10 in \cite{Mackey} and the discussion thereafter), whereas for the polydisk $\D^n$, $b\overline{\D}^n \subsetneq \partial \overline{\D}^n$. An interested reader is referred to \cite{ArvesonII} for further details. We recall a useful result in this direction.

\begin{prop}[\cite{Dales}, Proposition 2.4]\label{prop_Dales}
Let $\mathcal{A}$ be a finitely generated normed function algebra of $X$ which is inverse-closed. Then Shilov boundary for $\mathcal{A}$ is the closure of the set of peak points.	
\end{prop}

As an application of the above proposition, we have the following lemma for $b(Z(q) \cap \BC)$. From here onwards, we use the notation $cw=(cw_1, \dotsc, cw_n)$ for $w=(w_1, \dotsc, w_n) \in \C^n$ and $c \in \C$.

\begin{lem}\label{lem_Dales}
If $w \in b(Z(q) \cap \BC)$, then $-w \in b(Z(q) \cap \BC)$, where $q$ is as in $(\ref{eqn:new-T-001})$.
\end{lem}

\begin{proof}
The rational algebra $\mathcal{A}=Rat(Z(q) \cap \BC)$ is a normed function algebra which is also inverse-closed. We have by Lemma \ref{lem_pi} that 
\[
Z(q) \cap \BC=\left\{  \left(\frac{z}{\sqrt{1+r^2}},  \frac{rz^{-1}}{\sqrt{1+r^2}}\right) : r \leq |z|\leq 1\right\}.
\]
Let us define complex-valued functions $p_0, p_1$ on $Z(q) \cap \BC$ as follows:
\[
p_0\left(\frac{z}{\sqrt{1+r^2}},  \frac{rz^{-1}}{\sqrt{1+r^2}}\right):=\frac{z}{\sqrt{1+r^2}}, \quad p_1\left(\frac{z}{\sqrt{1+r^2}},  \frac{rz^{-1}}{\sqrt{1+r^2}}\right):=\frac{rz^{-1}}{\sqrt{1+r^2}}.
\]
Let $\mathcal{F}=\{p_0, p_1\}$ and let $\mathcal{A}_0$ be the smallest closed subalgebra of $\mathcal{A}$ containing $\mathcal{F}$ and the identity of $\mathcal{A}$ (which is the constant function $1$). Let $f \in \mathcal{A}$. By Lemma \ref{lem_pconvex}, there is a sequence of polynomials $\{f_n\}$ in two variables such that $\|f_n-f\|_{\infty, Z(q) \cap \BC} \to 0$ as $n \to \infty$. Note that the map
\[
f_n:Z(q) \cap \BC \to \mathbb{C}, \qquad \left(\frac{z}{\sqrt{1+r^2}},  \frac{rz^{-1}}{\sqrt{1+r^2}}\right) \mapsto f_n\left(\frac{z}{\sqrt{1+r^2}}, \frac{rz^{-1}}{\sqrt{1+r^2}}\right)
\]
can be written as polynomial in $z$ and $z^{-1}$. In other words, the restriction of each $f_n$ to $Z(q) \cap \BC$ is in $span \, \{p_0, p_1\}$. Hence, $f$ is in the closed span of $\mathcal{F}$ and so, $f \in \mathcal{A}_0$. Consequently, $\mathcal{A}_0 =\mathcal{A}$ and so, $\mathcal{A}$ is a finitely generated algebra. It follows from Proposition \ref{prop_Dales} that the set of peak points of $Z(q) \cap \BC$ is dense in $b(Z(q) \cap \BC)$. Let $w \in b(Z(q) \cap \BC)$. Then there is a sequence $\{w_n\}$ of peak points in $Z(q) \cap \BC$ such that ${\displaystyle \lim_{n \to \infty}w_n=w}$. Let $f_n$ be a peaking function for $w_n$, i.e., $f_n \in Rat(Z(q) \cap \BC), f_n(w_n)=1$ and $|f_n(z)|<1$ for all $z, w_n \in Z(q) \cap \BC$ with $z \ne w_n$. Then $g_n(z)=f_n(-z)$ is in $Rat(Z(q) \cap \BC)$ and $g_n$ peaks at $-w_n$. Hence, $\{-w_n\}$ is a sequence of peak points that converges to $-w$. By definition, $b(Z(q) \cap \BC)$ is a closed set that contains $\{-w_n\}$ and so, $-w \in b(Z(q) \cap \BC)$. The proof is now complete.
\end{proof}
 
 We now describe the distinguished boundary of the set $Z(q) \cap \BC$.

\begin{prop}\label{prop_bB}
$b(Z(q) \cap \BC)=Z(q) \cap \partial \B_2$, where $q$ is as in $(\ref{eqn:new-T-001})$.
\end{prop}

\begin{proof}
Let $w \in Z(q) \cap \partial \B_2$. Then $w=(w_1, w_2) \in \partial \B_2$. The function $f_w: \overline{\B}_2 \to \C $ defined by $f_w(z)=\langle z, w\rangle=z_1\overline{w}_1+ z_2\overline{w}_2$ for all $w=(w_1, w_2)$ in $\overline{\B}_2$ is holomorphic on $\BC$. By Cauchy-Schwarz inequality, $|f_w(z)| \leq 1$ and $|f_w(e^{i\theta}w)|=1$ for all $z \in \BC$ and $\theta \in \mathbb{R}$. Take $z=(z_1, z_2) \in \BC$ such that $|f_w(z)|=1$. Again, by Cauchy-Schwarz inequality, $1=|\langle z, w \rangle| \leq \|z\| \|w\| \leq 1$ which holds if and only if $z=c w$ for some $c \in \mathbb{T}$. Therefore, the maximum-modulus of $f_w$ occurs only at points of the form $e^{i\theta}w$ for $\theta \in \mathbb{R}$. Define $g_w=f_w|_{Z(q) \cap \BC}$ which is in $Rat(Z(q) \cap \BC)$ and $|g_w(z)| \leq 1$ for all $z \in Z(q) \cap \BC$. Moreover, $\pm w \in Z(q) \cap \BC$ and $|g_w(\pm w)|=1$. If  $|g_w(z)|=1$ for some $z \in Z(q) \cap \BC$, then $|f_w(z)|=1$ and so, $z=cw$ for some $|c|=1$. Any point of the form $cw \in Z(q) \cap \BC$ if and only if $c^2=1$. Hence, $g_w$ attains its maximum modulus at $\pm w$. Thus, either $w$ or $-w$ is in $b(Z(q) \cap \BC)$. In either case, we have by Lemma \ref{lem_Dales} that $w \in b(Z(q) \cap \BC)$ and so, $Z(q) \cap \partial \B_2 \subseteq b(Z(q) \cap \BC)$.

\smallskip

Let $f \in Rat(Z(q) \cap \BC)$. Then $g=f \circ \kappa \in Rat(\CA_r)$. By maximum modulus principle, there exists $z \in \partial \A_r$ such that $|g(z)|=\|g\|_{\infty, \CA_r}$. Let $w=\kappa(z)$. By  Lemma \ref{lem_pi}, $w \in Z(q) \cap \partial \B_2$. Then
\begin{equation*}
\begin{split}
|f(w)|=|g(z)|
&=\sup\{|g(\lambda)| : \lambda \in \CA_r \}\\
&=\sup\{|f(\kappa(\lambda))| : \lambda \in \CA_r \}\\
& = \sup\{|f(\kappa(\lambda))| : \kappa(\lambda) \in Z(q) \cap \BC\} \quad [\mbox{by Lemma \ref{lem_pi}}] \\
& = \|f\|_{\infty, Z(q) \cap \BC}. \\
\end{split}
\end{equation*}
Thus, $f$ attains its maximum modulus at some point of $Z(q) \cap  \partial \B_2$. Since $f \in Rat(Z(q) \cap \BC)$ is arbitrary, we have that $b(Z(q) \cap \BC) \subseteq Z(q) \cap \partial \B_2$.
\end{proof}

The next result presents an interplay between a constrained dilation on $\BC$ versus the principal variety being a complete spectral set.
\begin{thm}\label{thm_ArvI}
Let $(T_1, T_2)$ be a commuting pair of operators acting on a Hilbert space $\HS$. Then $(T_1, T_2)$ admits dilation to a $\B_2$-unitary $(U_1, U_2)$ with $q(U_1, U_2)=0$ if and only if  $Z(q) \cap \BC$ is a complete spectral set for $(T_1, T_2)$. 
\end{thm}

\begin{proof}
		Assume that	$(T_1, T_2)$ admits dilation to a $\B_2$-unitary $(U_1, U_2)$ on a space $\mathcal{K} \supseteq \mathcal{H}$ such that $q(U_1, U_2)=0$. By spectral mapping principle, $q(\sigma_T(U_1, U_2))=\sigma(q(U_1, U_2))=\{0\}$ and so, $\sigma_T(U_1, U_2) \subseteq Z(q) \cap \partial \B_2$. Thus, $Z(q) \cap \partial \B_2$ is a spectral set for $(U_1, U_2)$, because for commuting tuple of normal operators, its Taylor joint spectrum is a spectral set. Also, $f(T_1, T_2)=P_\mathcal{H}f(U_1, U_2)|_\mathcal{H}$ for every $f \in Rat(\BC)$ and hence for every polynomial $f$ in two variables. Since $Z(q) \cap \BC$ is a polynomially convex set, we have by Oka-Weil Theorem (see CH-7 in \cite{Oka-Weil}) that 
		\[
		f(T_1, T_2)=P_\mathcal{H}f(U_1, U_2)|_\mathcal{H}
		\]
for every $f \in Rat(Z(q) \cap \BC)$. By Proposition \ref{prop_bB}, $\sigma_T(U_1, U_2) \subseteq Z(q) \cap \partial \B_2=b(Z(q) \cap \BC)$. It follows from Theorem \ref{thm_Arveson} that $Z(q) \cap \BC$ is a complete spectral set for $(T_1, T_2)$.
		
\medskip

Conversely, let $Z(q) \cap \BC$ be a complete spectral set for $(T_1, T_2)$. Again by Theorem \ref{thm_Arveson}, there is a commuting pair of normal operators $(U_1, U_2)$ on a Hilbert space $\mathcal{K} \supseteq \mathcal{H}$ such that $\sigma_T(U_1, U_2)\subseteq b(Z(q)\cap \BC)$ and $f(T_1, T_2)=P_\mathcal{H}f(U_1, U_2)|_\mathcal{H}$ for every polynomial $f$ in two variables. It follows from Proposition \ref{prop_bB} that $b(Z(q) \cap \BC)=Z(q) \cap \partial\B_2$ and so, $\sigma_T(U_1, U_2) \subseteq Z(q) \cap \partial \B_2$. Since $q(U_1, U_2)$ is a normal operator, we have by spectral mapping theorem that
\[
\|q(U_1, U_2)\|=\{|\lambda| : \lambda \in \sigma(q(U_1, U_2))\}=\{|\lambda| : \lambda \in q(\sigma_T(U_1, U_2))\}=0.
\]
Therefore, $(U_1, U_2)$ is a $\B_2$-unitary with $q(U_1, U_2)=0$ and $p(T_1, T_2)=P_\HS p(U_1, U_2)|_{\HS}$ for all polynomials $p \in \C[z_1, z_2]$. Since $\BC$ is polynomially convex, every $f \in Rat(\BC)$ can be approximated by polynomials by Oka-Weil theorem (see CH-7 of \cite{Oka-Weil}) and it follows that  $f(T_1, T_2)=P_\HS f(U_1, U_2)|_{\HS}$ for all $f \in Rat(\BC)$. The proof is now complete.
	\end{proof}

We have already seen in Theorem \ref{thm_A_con_Ball} that $Z(q) \cap \BC$ is a spectral set for $\kappa(T)$ when $T$ is an $\A_r$-contraction and vice-versa. Since Theorem \ref{thm_Agler} tells us that for an $\A_r$-contraction $T$, $\CA_r$ is a complete spectral set, it is natural to ask if $Z(q) \cap \BC$ is a complete spectral set for $\kappa(T)$ when $T$ is an $\A_r$-contraction. The next theorem gives an affirmative answer to this.

\begin{thm}\label{thm_A_con_BallII}
Let $T$ be an invertible operator acting on a Hilbert space $\HS$. Then the following are equivalent:

\begin{enumerate}
\item $T$ is an $\A_r$-contraction;
\item $Z(q) \cap \BC$ is a complete spectral set for $\kappa(T)$;
\item $Z(q) \cap \BC$ is a spectral set for $\kappa(T)$.
\end{enumerate}
\end{thm}

\begin{proof}
$(2) \implies (3)$ is trivial and $(3) \implies (1)$ follows from Theorem \ref{thm_A_con_Ball}. We prove $(1) \implies (2)$. Let $T$ be an $\A$-contraction on $\HS$. It follows from Theorem \ref{thm_Agler} that there is an $\A_r$-unitary $U$ on a larger Hilbert space $\KS$ containing $\HS$ such that $f(T)=P_\HS f(U)|_{\HS}$ for all $f \in Rat(\CA_r)$. We have by Proposition \ref{prop_A_uni_Ball} that 
\[
(U_1, U_2)=\kappa(U)=\left(\frac{U}{\sqrt{1+r^2}}, \frac{rU^{-1}}{\sqrt{1+r^2}} \right)
\]
is a $\B_2$-unitary. It is evident that $q(U_1, U_2)=0$. Let $(T_1, T_2)=\kappa(T)$ and let $f \in Rat(\BC)$. Then $f\circ \kappa \in Rat(\CA_r)$ and we have
\[
f(T_1, T_2)=f\circ \kappa(T)=P_\HS f \circ \kappa(U)|_\HS=P_\HS  f(U_1, U_2)|_\HS.
\]
Thus, $(T_1, T_2)$ admits dilation to $\B_2$-unitary $(U_1, U_2)$ with $q(U_1, U_2)=0$. The desired conclusion now follows from Theorem \ref{thm_ArvI} and the proof is complete.
\end{proof}

We present a few characterizations of $\A_r$-contractions that are simple consequences of Theorem \ref{thm_Agler}. Recall that an $\A_r$-coisometry is an invertible operator whose adjoint is an $\A_r$-isometry.

\begin{prop}\label{thm_coiso}
An invertible operator $T$ on a Hilbert space $\HS$ is an $\A_r$-contraction if and only if there is an $\A_r$-coisometry $V$ on a Hilbert space $\mathcal{K} \supseteq \HS$ such that $V\HS \subseteq \HS$ and $T=V|_\HS$. 
\end{prop}

\begin{proof}
We assume that $T$ is an $\A_r$-contraction. By definition, $T^*$ is an $\A_r$-contraction. It follows from Theorem \ref{thm_Agler} that there exists a larger Hilbert space $\mathcal{K}_0 \supseteq \HS$ and an $\A_r$-unitary $U$ on $\mathcal{K}_0$ such that $T^{*m}=P_\HS U^m|_\HS$ for all $m \in \Z$. Set 
\[
\mathcal{K}:=\overline{span}\{U^mh : h \in \HS, m \in \Z\},
\]
which is invariant under $U, U^{-1}$. Let $V=U|_{\mathcal{K}}$. Clearly, $V$ is invertible and $\mathcal{K}=\overline{span}\{V^mh : h \in \HS, m \in \Z\}$. By definition, $V$ is an $\A_r$-isometry on $\mathcal{K}$. Let $m, j \in \Z$ and $h \in \HS$. Then 
\[
T^{*m}P_\HS V^jh=T^{*(m+j)}h=P_\HS V^{m+j}h=P_\HS V^mV^jh \quad \text{and so,} \quad T^{*m}P_\HS=P_\HS V^m.
\] 
We have that $\la T^mh, h' \ra=\la P_\HS T^m h, h' \ra=\la V^{*m} P_\HS h, h'\ra=\la V^{*m}h, h' \ra$ for all $h' \in \mathcal{K}$ and thus $T^m=V^{*m}|_\HS$. Conversely, suppose $T=V|_\HS$ for an $\A_r$-coisometry $V$. Since $T$ and $V$ are both invertible, it follows that $T^{-1}=V^{-1}|_\HS$ and so, $T^m=V^m|_\HS$ for $m \in \Z$. Since $V$ is an $\A_r$-contraction, $T$ also becomes an $\A_r$-contraction which completes the proof.
\end{proof}

\begin{thm}\label{thm_A_model}
Let $T$ be an invertible operator on a Hilbert space $\HS$. Then $T$ is an $\A_r$-contraction if and only if there exist a Hilbert space $\mathcal{L}$, an operator $A: \mathcal{L} \to \HS$ and an $\A_r$-coisometry $B \in \mathcal{B}(\mathcal{L})$ such that for $B_1=(1+r^2)^{-1\slash2}B$,
\[
\alpha(T^*, T)=(1+r^2)^2AB^{-*}(I-B_1^*B_1)^2B^{-1}A^* \quad \& \quad T^*A+AB=(1+r^2)AB^{-*}.
\]
\end{thm}

\begin{proof}
Let $T$ be an $\A_r$-contraction. Then $T^*$ is also an $\A_r$-contraction. By Proposition \ref{thm_coiso}, there is an $\A_r$-coisometry $V$ on a larger Hilbert space $\mathcal{K}$ containing $\HS$ such that $T^{*m}=V^m|_\HS$ for all $m \in \Z$. Thus, we can write $V=\begin{bmatrix}
T^* & A\\
0 & B
\end{bmatrix}$ and $V^{-1}=\begin{bmatrix}
T^{-*} & A_0\\
0 & B_0
\end{bmatrix}$ as the $2 \times 2$ blocks of operators with respect to $\mathcal{K}=\HS \oplus \HS^\perp$. Thus, $B$ is invertible with $B^{-1}=B_0$. Let $C=T^*A+AB$. Then,
\begin{align}\label{eqn_A_modelI}
\alpha(V, V^*)
&=-V^2V^{*2}+(1+r^2)VV^*-r^2I \notag \\
&=-\begin{bmatrix}
T^{*2} & C\\
0 & B^2
\end{bmatrix}
\begin{bmatrix}
T^2 & 0\\
C^* & B^{*2}
\end{bmatrix}+(1+r^2)\begin{bmatrix}
T^* & A\\
0 & B
\end{bmatrix}\begin{bmatrix}
T & 0\\
A^* & B^*
\end{bmatrix}-r^2\begin{bmatrix}
I & 0\\
0 & I
\end{bmatrix} \notag \\
&=\begin{bmatrix}
\alpha(T^*, T)+(1+r^2)AA^*-CC^* & (-CB^*+(1+r^2)A)B^*\\
B(-BC^*+(1+r^2)A^*) & \alpha(B, B^*)
\end{bmatrix}.
\end{align}
It follows from Theorem \ref{thm_A_iso} that $V^*$ is an $\A_r$-isometry if and only if the following holds: 
\begin{equation}\label{eqn_A_modelII}
\alpha(T^*, T)=-(1+r^2)AA^*+CC^*, \quad (-CB^*+(1+r^2)A)B^*=0 \quad \text{and} \quad \alpha(B, B^*)=0.
\end{equation}
By Theorem \ref{thm_A_iso} and (\ref{eqn_A_modelII}), $B^*$ is an $\A_r$-isometry on $\HS^\perp$. Again by (\ref{eqn_A_modelII}), $(1+r^2)A=CB^*$ and so, $C=T^*A+AB=(1+r^2)AB^{-*}$. Note that
\begin{align}
\alpha(T^*, T)
&=CC^*-(1+r^2)AA^* \quad [\text{by} \ (\ref{eqn_A_modelII})] \notag \\
&=(1+r^2)^2AB^{-*}B^{-1}A^*-(1+r^2)AA^* \qquad [\text{since } \  C=-(1+r^2)AB^{-*}] \notag \\
&=(1+r^2)AB^{-*}\left((1+r^2)I-B^*B\right)B^{-1}A^* \notag \\
&=(1+r^2)^2AB^{-*}(I-B_1^*B_1)^2B^{-1}A^*,
\end{align}
where $B_1=(1+r^2)^{-1\slash2}B$. Choose $\mathcal{L}=\HS^\perp$ and the forward direction follows. 

\smallskip 

Conversely, assume that there exist a space $\mathcal{L}$, an operator $A: \mathcal{L} \to \HS$ and an $\A_r$-coisometry $B \in \mathcal{B}(\mathcal{L})$ such that $\alpha(T^*, T)=(1+r^2)^2AB^{-*}(I-B_1^*B_1)^2B^{-1}A^*$ and $T^*A+AB=(1+r^2)AB^{-*}$ for $B_1=(1+r^2)^{-1\slash2}B$. Suppose 
\[
V=\begin{bmatrix}
T^* & A\\
0 & B
\end{bmatrix} \ \ \text{on} \ \ \mathcal{K}=\HS \oplus \mathcal{L} \quad  \text{so that } \quad V^{-1}=\begin{bmatrix}
T^{-*} & -T^{-*}AB^{-1} \\
0 & B^{-1}
\end{bmatrix}.
\] 
This shows that $T^{*m}=V^m|_\HS$ for all $m \in \Z$. Following the computations in (\ref{eqn_A_modelI}) and (\ref{eqn_A_modelII}), we have by Theorem \ref{thm_A_iso} that $V$ is an $\A_r$-coisometry. Therefore, $T^*$ and thus, $T$ is an $\A_r$-contraction and the proof is complete. 
\end{proof}

\section{The $C_\alpha$ class and spherical contractions}\label{sec_spher}

\noindent Recall that the class $C_\alpha$ consists of invertible contractions $T$ satisfying $\alpha(T^*,T)\geq 0$, where
\[
\alpha(T^*, T)=-T^{*2}T^2+(1+r^2)T^*T-r^2I.
\]
The $C_\alpha$ class was explored by Bello \& Yakubovich in \cite{Bello} and a model theorem for this class was obtained there. Moreover, it was proved in \cite{Bello} that the class of $\A_r$-contractions is strictly smaller than the $C_\alpha$ class. More precisely, the following theorem was proved.
\begin{thm}[\cite{Bello}, Theorem 1.1]\label{thm_Bello_1.1}
Every $\A_r$-contraction is in $C_\alpha$.
\end{thm}
In this section, we study the $C_\alpha$ class  in view of spherical contractions. It becomes possible because of the following straightforward result.
\begin{prop}\label{prop_C_sph}
An invertible operator $T$ is in $C_\alpha$ if and only if the induced pair $\kappa(T)$, where
\[
\kappa(T)=\left(\frac{T}{\sqrt{1+r^2}}, \frac{rT^{-1}}{\sqrt{1+r^2}} \right),
\] 
is a spherical contraction. 
\end{prop}

\begin{proof}
Let $T$ be an invertible operator and let $\kappa(T)=(T_1, T_2)$. Note that $\alpha(T^*, T) \geq 0$ if and only if $T^{-*}\alpha(T^*, T)T^{-1} \geq 0$. A routine computation shows that
\[
T^{-*}\alpha(T^*, T)T^{-1}=(1+r^2)\left[I-\frac{T^*T}{1+r^2}-\frac{r^2T^{-*}T^{-1}}{1+r^2}\right]=(1+r^2)(I-T_1^*T_1-T_2^*T_2).
\]
Consequently, $T$ is in $C_\alpha$ if and only if $\kappa(T)$ is a spherical contraction.  
\end{proof}

As an application of the theory of spherical contractions and the relation between $\mathbb{B}_2$-contractions and $\A_r$-contractions as discussed in Section \ref{sec_var}, we give several independent proofs to Theorem \ref{thm_Bello_1.1}. 

\smallskip

\noindent \textbf{\textit{Proofs of Theorem \ref{thm_Bello_1.1}.}} We present three different proofs of Theorem \ref{thm_Bello_1.1}. Let $T$ be an $\A_r$-contraction acting on a Hilbert space $\HS$.

\smallskip

\noindent \textit{Proof $1$.} The first proof utilizes the theory of spherical contractions. By Theorem \ref{thm_Agler}, there is a Hilbert space $\mathcal{K}$ containing $\HS$ and an $\A_r$-unitary $U$ on $\mathcal{K}$ such that $T^n=P_\HS U^n|_\HS$ for all $n \in \Z$. Let us define 
\[
(T_1, T_2)=\kappa(T)=\left(\frac{T}{\sqrt{1+r^2}}, \frac{rT^{-1}}{\sqrt{1+r^2}} \right) \quad \text{and} \quad (U_1, U_2)=\kappa(U)=\left(\frac{U}{\sqrt{1+r^2}}, \frac{rU^{-1}}{\sqrt{1+r^2}} \right).
\]
It follows from Proposition \ref{prop_A_uni_Ball} that $(U_1, U_2)$ is a $\B_2$-unitary and so, by Theorem \ref{thm_Bn_iso}, $U_1^*U_1+U_2^*U_2=I$. Therefore, $\|U_1y\|^2+\|U_2y\|^2=\|y\|^2$ for all $y \in \mathcal{K}$. Since $T_i=P_\HS U_i|_\HS$ for $i=1,2$, we have that 
\[
\|T_1x\|^2+\|T_2x\|^2 \leq \|U_1x\|^2+\|U_2x\|^2=\|x\|^2
\]
for all $x \in \HS$. Thus, $(T_1, T_2)$ is a spherical contraction. We have by Proposition \ref{prop_C_sph} that $T \in C_\alpha$. \qed 

\medskip 

\noindent \textit{Proof $2$.} Here, we use Proposition \ref{thm_coiso} which states that every $\A_r$-contraction is a restriction of an $\A_r$-coisometry to an invariant subspace. Consequently, there exists an $\A_r$-coisometry $V$ on a space $\mathcal{K}$ containing $\HS$ such that $T^m=V^m|_\HS$ for all $m \in \Z$. We show that $V \in C_\alpha$. Since $V^*$ is an $\A_r$-isometry, it has an extension to an $\A_r$-unitary $N$ acting on $\mathcal{K}'$. With respect to $\mathcal{K}'=\mathcal{K} \oplus \mathcal{K}^\perp$, we can write 
\[
N=\begin{bmatrix}
V^* & A \\
0 & B
\end{bmatrix}.
\]
Since $N$ is normal, we have that 
\[
N^*N=\begin{bmatrix}
VV^* & VA \\
A^*V^* & A^*A+B^*B
\end{bmatrix}, \quad  NN^*=\begin{bmatrix}
V^*V+AA^* & AB^* \\
BA^* & BB^*
\end{bmatrix} \quad \text{and so,} \quad VA=AB^*.
\]
By Proposition \ref{prop_A_unit}, we have that
\begin{equation*}
\begin{split}
0=(I-N^*N)(N^*N-r^2I)
&=-N^2N^{*2}+(1+r^2)NN^*-r^2I\\
&=\begin{bmatrix}
-V^{*2}V^2-CC^*+(1+r^2)(V^*V+AA^*)-r^2I & * \\
* & * 
\end{bmatrix}\\
&=\begin{bmatrix}
\alpha(V^*, V)-CC^*+(1+r^2)AA^* & * \\
* & * 
\end{bmatrix},
\end{split}
\end{equation*}
where $C=V^*A+AB$, and that 
\begin{equation*}
\begin{split}
0=(I-N^*N)(N^*N-r^2I)
=-N^{*2}N^{2}+(1+r^2)N^*N-r^2I=\begin{bmatrix}
* & -V^2C+(1+r^2)VA \\
* & * 
\end{bmatrix}.
\end{split}
\end{equation*}
Combining these things together, we have that
\begin{equation}\label{eqn_802}
(i) \ VA=A^*B, \quad (ii) \ \alpha(V^*, V)=CC^*-(1+r^2)AA^* \quad \text{and} \quad (iii) \ (1+r^2)VA=V^2C.
\end{equation}
Note that $\|B\| \leq \|N\| \leq 1 \leq  \sqrt{1+r^2}$ and so, $(1+r^2) -B^*B \geq 0$. Then
\begin{equation*}
\begin{split}
V\alpha(V^*, V)V^*&=VCC^*V^*-(1+r^2)VAA^*V^* \qquad \quad \ \ \ [\text{by} \ (\ref{eqn_802})(ii)]\\
&=(1+r^2)^2AA^*-(1+r^2)A^*BB^*A \qquad [\text{by} \ (\ref{eqn_802})(iii) \ \& \ (i)]\\
&=(1+r^2)A((1+r^2)-B^*B)A^*\\
& \geq 0.
\end{split}
\end{equation*}
As $V$ is invertible, $\alpha(V^*, V) \geq 0$ and so, $V \in C_\alpha$. Thus, $T \in C_\alpha$ as $T$ is invertible and $T=V|_\HS$. \qed

\medskip 

\noindent \textit{Proof $3$.} It follows from Theorem \ref{thm_A_model} that there is a Hilbert space $\mathcal{L}$, an operator $A: \mathcal{L} \to \HS$ and an $\A_r$-coisometry $B \in \mathcal{B}(\mathcal{L})$ such that 
$
\alpha(T^*, T)=(1+r^2)^2AB^{-*}(I-B_1^*B_1)^2B^{-1}A^*=P^*P,
 $
where $B_1=(1+r^2)^{-1\slash2}B$ and $P=(1+r^2)(I-B_1^*B_1)B^{-1}A^*$. Hence, $T \in C_\alpha$. \qed 

\bigskip

We recall from \cite{Bello} the explicit construction of a weighted bilateral shift in $C_\alpha$. For $0<r<1$, we denote by $\Omega_r$ the complement of $\overline{\A}_r$, i.e., $\Omega_r=\{z \in \C : |z|<r \ \text{ or } \ |z|>1\}$.
For a Hilbert space $\mathcal{E}$, the space $H^2(\Omega_r, \mathcal{E})$ is the space all holomorphic functions $f: \Omega_r \to \mathcal{E}$ which can be represented as
\begin{equation*}
f(z)=\overset{\infty}{\underset{n=0}{\sum}}f_nz^n \quad \text{for} \ |z| < r, \qquad f(z)=\overset{-1}{\underset{n=-\infty}{\sum}}f_nz^n \quad \text{for} \ |z| >1, \quad(\{f_n\}_{n \in \Z} \subset \mathcal{E})
\end{equation*}
and is equipped with the norm
\[
\|f\|^2_{H^2(\Omega_r, \mathcal{E})}=\frac{1}{1-r^2} \overset{\infty}{\underset{n=0}{\sum}}\|r^nf_n\|^2 + \frac{1}{1-r^2} \overset{-1}{\underset{n=-\infty}{\sum}}\|f_n\|^2.
\]
 We simply write $f=\{f_n\}$ for $f \in H^2(\Omega_r, \mathcal{E})$.  Moreover, $H^2(\Omega_r, \mathcal{E})$ is a Hilbert space with the inner product
$\displaystyle
\langle f,g \rangle= \frac{1}{1-r^2} \left[\overset{\infty}{\underset{n=0}{\sum}}\langle r^nf_n, r^ng_n \rangle + \overset{-1}{\underset{n=-\infty}{\sum}} \langle f_n, g_n \rangle \right]$ for  $f=\{f_n\}, g=\{g_n\}$ in $H^2(\Omega_r, \mathcal E)$. We define an operator $M_z^t$ on  $H^2(\Omega_r, \mathcal E)$ as
\[
M_z^t(f)(z)=zf(z)-(zf(z))|_{z=\infty} \ ,
\]
where ${ \displaystyle (zf(z))|_{z=\infty}=\lim_{z \to \infty}zf(z) }$ and the limit is finite. Evidently, if one identifies  $f=\{f_n\}$ in $H^2(\Omega_r, \mathcal{E})$ with the two-sided vector sequence $(\dotsc,f_{-2},f_{-1},\boxed{f_0}, f_1,f_2,\dotsc)$, then $M_z^t$ takes the form of the weighted bilateral shift given by
\[
M_z^t(\dotsc, f_{-2}, f_{-1}, \boxed{f_0}, f_1, f_2, \dotsc )=(\dotsc, f_{-3}, f_{-2}, \boxed{-f_{-1}}, f_0, f_1, \dotsc) \ ,
\]
where the box indicates the $0$-th position and the weight at the $0$-th position is $-1$. Here we shall use this identification of $M_z^t$ as the weighted bilateral shift throughout the sequel. Under this identification, the inverse on $M_z^t$ is given by
\[
(M_z^t)^{-1}(\dotsc, f_{-2}, f_{-1}, \boxed{f_0}, f_1, f_2, \dotsc )=(\dotsc, f_{-1}, -f_{0}, \boxed{f_{1}}, f_2, f_2, \dotsc).
\]
 The operator $M_z^t$ on $H^2(\Omega_r, \mathcal{E})$ also serves as an example of an operator in $C_\alpha$ which is not an $\A_r$-contraction, e.g., see \cite{Bello}. Here we mention another interesting property of the operator $M_z^t$ on $H^2(\Omega_r, \mathcal{E})$ which will be useful later. Indeed, we show that there is no $s \in (0, 1)$ such that $M_z^t$ on $H^2(\Omega_r, \mathcal{E})$ becomes an $\A_s$-contraction.
 
\begin{eg}\label{eg_5}
 Let us assume the contrary, i.e., $M_z^t$ on $H^2(\Omega_r, \mathcal{E})$ is an $\A_s$-contraction for some $s \in (0,1)$. Then, $M_z^{t*}$ is also an $\A_s$-contraction. A simple calculation gives that
\[
	M_z^{t*}(\dotsc, f_{-2}, f_{-1}, \boxed{f_0}, f_1, f_2, \dotsc )=(\dotsc, f_{-1}, -f_{0}, \boxed{r^2f_{1}}, r^2f_2, r^2f_3, \dotsc)
\]
for all $f=\{f_n\}_{n \in \mathbb{Z}} \in H^2(\Omega_r, \mathcal{E})$. Here the box represents the zeroth position. We have by  Theorem \ref{thm_Bello_1.1} that $M_z^{t*}$ satisfies $(1+s^2)\|M_z^{t*}f\|^2-\|(M_z^{t*})^2f\|^2-s^2\|f\|^2 \geq 0$ for all $f \in H^2(\Omega_r, \mathcal{E})$. Fix some unit vector $u \in \mathcal{E}$ and define $f=(\dotsc, 0, \boxed{0}, u, 0,  \dotsc)$. 
Then
\[
M_z^{t*}f=(\dotsc, 0, \boxed{r^2u}, 0, \dotsc), \quad (M_z^{t*})^2f=(\dotsc,0,  -r^2u, \boxed{0}, 0, \dotsc)  
\]
and so, $\displaystyle \|M_z^{t*}f\|^2=\|(M_z^{t*})^2f\|^2=\frac{r^4}{1-r^2}$. Consequently, we have that
\[
(1+s^2)\|M_z^{t*}f\|^2-\|(M_z^{t*})^2f\|^2-s^2\|f\|^2=\frac{(1+s^2)r^4}{1-r^2}-\frac{r^4}{1-r^2}-\frac{s^2r^2}{1-r^2}=-s^2r^2<0,
\] 
a contradiction. Hence, $M_z^{t}$ cannot be an $\A_s$-contraction for any $0<s<1$. Also, $(M_z^t)^* \notin C_\alpha$. \qed
\end{eg}

Theorem \ref{thm_A_con_BallII} connects an $\A_r$-contraction $T$ with the induced $\B_2$-contraction $\kappa(T)$ which further has the principal variety $\BC \cap Z(q)$ as a spectral set. Also, Proposition \ref{prop_C_sph} shows that $T$ is in $C_\alpha$ class if and only if $\kappa(T)$ is a spherical contraction. Interestingly, for $T \in C_{\alpha}$ the induced pair $\kappa(T)$ can also become a $\B_2$-contraction under certain condition. We shall fetch such a condition from \cite{AthavaleIII, EschmeierII, Muller}. To do so, we recall a few terminologies from \cite{Muller}. For a commuting pair of operators $\underline{T}=(T_1, T_2)$ acting on a Hilbert space $\HS$, let us define 
\[
M_{\underline{T}}: \mathcal{B}(\HS) \to \mathcal{B}(\HS), \quad A \mapsto T_1^*AT_1+T_2^*AT_2 \quad \& \quad 
\Delta_{\underline{T}}^{(n)}=(I-M_{\underline{T}})^n(I),
\]
where $n \in \N \cup \{0\}$. Evidently, $\underline{T}$ is a spherical contraction if and only if $\Delta_{\underline{T}}^{(1)} \geq 0$. The following result gives a sufficient condition for a commuting pair of operator to dilate to a $\mathbb{B}_2$-unitary.

\begin{thm}[\cite{Muller}, Theorem 11]\label{thm_Muller}
A pair $\underline{T}=(T_1, T_2)$ of commuting operators satisfying $\Delta_{\underline{T}}^{(1)} \geq 0$ and $\Delta_{\underline{T}}^{(2)} \geq 0$ dilates to a $\B_2$-unitary.
\end{thm}

An immediate consequence of the above theorem is the following proposition. 

\begin{prop}\label{prop_Muller}
Let $T \in C_\alpha$ and let $\beta(T^*, T) \geq 0$, where 
\[
\beta(T^*, T)=T^{*4}T^4-2(1+r^2)T^{*3}T^3+(1+r^4+4r^2)T^{*2}T^2-2r^2(1+r^2)T^*T+r^4I.
\] 
Then $\kappa(T)$ is a $\B_2$-contraction that admits dilation to a $\B_2$-unitary.
\end{prop}

\begin{proof}
Define 
\[
\underline{T}=(T_1, T_2)=\kappa(T)=\left(\frac{T}{\sqrt{1+r^2}}, \frac{rT^{-1}}{\sqrt{1+r^2}} \right).
\]
Then 
\begin{equation*}
	\begin{split}
\Delta_{\underline{T}}^{(2)}
=	(I-M_{\underline{T}})^2(I)
	=I+M_{\underline{T}}^2(I)-2M_{\underline{T}}(I)
&		=I+T_1^{*2}T_1^2+T_2^{*2}T_2^2+2T_1^*T_2^*T_1T_2-2T_1^*T_1-2T_2^*T_2.
	\end{split}
\end{equation*}
A routine calculation shows that
\[
(1+r^2)^2T^{*2}\Delta_{\underline{T}}^{(2)}T^2=T^{*4}T^4-2(1+r^2)T^{*3}T^3+(1+r^4+4r^2)T^{*2}T^2-2r^2(1+r^2)T^*T+r^4I,
\]
which is equal to $\beta(T^*, T)$. Since $T$ is invertible, it follows that $\Delta_{\underline{T}}^{(2)} \geq 0$ if and only if $\beta(T^*, T) \geq 0$. We have by Theorem \ref{thm_Muller} that $\kappa(T)$ is a $\B_2$-contraction that admits dilation to a $\B_2$-unitary. 	
\end{proof}

We now show that the converse to Theorem \ref{thm_Muller} and Proposition \ref{prop_Muller} may not hold. In fact, we prove that the strict inequality $\beta(T^*, T)>0$ is not valid even if $T$ is an $\A_r$-contraction. For this purpose, we recall from \cite{Agler} and \cite{Bello} the construction of another weighted bilateral shift on a vector valued Hardy space associated with an annulus. Given a Hilbert space $E$, the vector valued Hardy space $H^2(\A_r, E)$ is the space of all holomorphic maps $f: \A_r \to E$ such that
\[
\sup_{r < s <1} \frac{1}{2\pi} \int_0^{2\pi}\|f(se^{it})\|^2dt < \infty.
\]
For any positive and invertible operator $\omega$ on $E$, $H^2(\A_r, E, \omega)$  is defined as the vector space $H^2(\A_r, E)$ endowed with the Hilbert space norm
\[
\|f\|_{H^2(\A_r, E, \omega)}^2=\frac{1}{2\pi} \int_0^{2\pi}\|f(e^{it})\|^2dt+\frac{1}{2\pi} \int_0^{2\pi}\|\omega f(re^{it})\|^2dt.
\]
When $f \in H^2(\A_r, E, \omega)$, it is understood that $\|f\|=\|f\|_{H^2(\A_r, E, \omega)}$. Define the operator
\[
M_z: H^2(\A_r, E, \omega) \to H^2(\A_r, E, \omega), \quad M_z(f)(z)=zf(z)
\]
which has inverse $M_z^{-1}(f)(z)=z^{-1}f(z)$. For $f \in H^2(\A_r, E, \omega)$, it is not difficult to see that
\begin{align}\label{eqn_eg_pure}
&\quad \|M_zf\|^2
=\frac{1}{2\pi} \int_0^{2\pi}\|f(e^{it})\|^2dt+r^2\frac{1}{2\pi} \int_0^{2\pi}\|\omega f(re^{it})\|^2dt
<\|f\|^2, \notag \\
&\quad \|M_z^{-1}f\|^2
=\frac{1}{2\pi} \int_0^{2\pi}\|f(e^{it})\|^2dt+r^{-2}\frac{1}{2\pi} \int_0^{2\pi}\|\omega f(re^{it})\|^2dt
<r^{-2}\|f\|^2
\end{align}
and so, $M_z$ and $rM_z^{-1}$ are contractions. By spectral mapping theorem, $\sigma(M_z) \subseteq \CA_r$. We now show that $M_z$ is an $\A_r$-isometry. Let $f \in H^2(\A_r, E, w)$. For the ease of computations, we denote by
\[
a(f)=\frac{1}{2\pi} \int_0^{2\pi}\|f(e^{it})\|^2dt \quad \text{and} \quad b(f)=\frac{1}{2\pi} \int_0^{2\pi}\|\omega f(re^{it})\|^2dt.
\]
We have by (\ref{eqn_eg_pure}) that
\[
\|f\|^2=a(f)+b(f), \quad \|M_zf\|^2=a(f)+r^2b(f) \quad \text{and} \quad \|M_z^{-1}f\|^2=a(f)+r^{-2}b(f)
\]
and that
\[
(1+r^2)\|f\|^2-\|M_zf\|^2-\|rM_z^{-1}f\|^2=(1+r^2)(a(f)+b(f))-a(f)-r^2b(f)-r^2a(f)-b(f)=0.
\]
Consequently, it follows from Theorem \ref{thm_A_iso} that $M_z$ on $H^2(\A_r, E, \omega)$ is an $\A_r$-isometry.

\begin{eg}\label{eg_6}
 Let $a>0$. Following the discussion above, we consider the operator $T=M_z^*$ on the Hilbert space $H^2(\A_r, \C, aI)$ and let $\underline{T}=\kappa(T)$. Then, $T$ is an $\A_r$-contraction and thus is in $C_\alpha$. By Proposition \ref{prop_C_sph}, $\underline{T}$ is a spherical contraction and thus $\Delta_T^{(1)} \geq 0$.  It follows from Theorem \ref{thm_A_con_BallII} that $\underline{T}$ is a $\B_2$-contraction and it admits a dilation to a $\B_2$-unitary. Choose $(a, r)=(0.37, 0.75)$. We show that $\beta(T^*, T) \geq  0$ is not true which is same as saying that $\Delta_{\underline{T}}^{(2)} \geq 0$ is not true since $\beta(T^*, T)=(1+r^2)^2T^{*2}\Delta_{\underline{T}}^{(2)}T^2$. To do so, we first identify the space as 
\[
H^2(\A_r, \C, aI)=\left\{ f=\{f_n\}_{n \in \Z} \subset \C: \|f\|^2=\overset{\infty}{\underset{n=-\infty}{\sum}}|f_n|^2(1+a^2r^{2n}) < \infty \right\}.
\]
Clearly, $M_z\{f_n\}_{n \in \Z}=\{f_{n-1}\}_{n \in \Z}$. For $m\in \N$, we have
\[
T\{f_n\}_{n \in \Z}=\left\{\left(\frac{1+a^2r^{2(n+1)}}{1+a^2r^{2n}}\right)f_{n+1}\right\}_{n \in \Z} \ \ \text{and so,} \ \ T^m\{f_n\}_{n \in \Z}=\left\{\left(\frac{1+a^2r^{2(n+m)}}{1+a^2r^{2n}}\right)f_{n+m}\right\}_{n \in \Z}.
\]
Let $\alpha_n(m)=(1+a^2r^{2n})(1+a^2r^{2(n-m)})^{-1}$ for $m \in \mathbb{N}$ and $n \in \Z$. Consequently, we have that
\[
\|T^mf\|^2=\overset{\infty}{\underset{n=-\infty}{\sum}}|f_n|^2(1+a^2r^{2n})\left( \frac{1+a^2r^{2n}}{1+a^2r^{2(n-m)}}\right)=\overset{\infty}{\underset{n=-\infty}{\sum}}|f_{n}\alpha_{n} (m)|^2(1+a^2r^{2n}).
\]
Choose $f=\{f_n\}_{n \in \Z}$ such that $f_1=1$ and $f_n=0$ for $n \in \Z$ with $n \ne 1$. Let if possible, we have 
\begin{equation*}
\begin{split}
0 &  \leq \la \beta(T^*, T)f, f \ra\\
&=\|T^4f\|^2-2(1+r^2)\|T^3f\|^2+(1+r^4+4r^2)\|T^2f\|^2-2r^2(1+r^2)\|Tf\|^2+r^4\|f\|^2\\
&=(1+a^2r^2)\left(\alpha_1(4)-2(1+r^2)\alpha_1(3)+(1+r^4+4r^2)\alpha_1(2)-2r^2(1+r^2)\alpha_1(1)+r^4\right).
\end{split}
\end{equation*}
Then 
\begin{equation}\label{eqn_eg_6}
\alpha_1(4)-2(1+r^2)\alpha_1(3)+(1+r^4+4r^2)\alpha_1(2)-2r^2(1+r^2)\alpha_1(1)+r^4 \geq 0.
\end{equation}
A few steps of routine calculations show that \eqref{eqn_eg_6} is not true for $(a, r)=(0.37, 0.75)$. Hence, the converse to Proposition \ref{prop_Muller} does not hold. Also, it shows that the converse of Theorem \ref{thm_Muller} is not true in general.
\qed
\end{eg}

\section{Minimal spectral set and minimal von Neumann set for operators associated with an annulus}\label{sec_vNset}


\noindent The authors of \cite{Bello} proved that the operators in $C_\alpha$ and $C_{1,r}$ do not necessarily have $\CA_r$ as a spectral set. Rather, they obtained a bound $K$ such that the operators in $C_{\alpha}$ satisfy von Neumann's inequality (\ref{eqn:new-002}) as stated below.

\begin{thm}[\cite{Bello},Theorem 1.3]\label{thm_Bello_1.3}
$\CA_r$ is a complete $\sqrt{2}$-spectral set for all operators in $C_\alpha$.
\end{thm}

Tsikalas \cite{Tsikalas} further showed that the constant $\sqrt{2}$ is optimal for $C_{\alpha}$ in the sense that there is no positive real number $K< \sqrt 2$ such that $\CA_r$ is a complete $K$-spectral set for $C_\alpha$. Also, the following bounds were found in \cite{Crou, BBC} for the $C_{1, r}$ class.

\begin{thm}[\cite{Crou},Theorem 9 \& \cite{BBC}, Theorem 1.2]\label{thm_Crou_9}
$\CA_r$ is a $(1+\sqrt{2})$-spectral set for all operators in $C_{1,r}$ and $\CA_r$ is a complete $(2+2\slash \sqrt{3})$-spectral set for all operators in $C_{1,r}$.
\end{thm}

Again, Tsikalas \cite{TsikalasII} showed that if $\CA_r$ is a $K$-spectral set for all operators in $C_{1,r}$ for some $K>0$, then $K\geq 2$. Whether the bound $2$ is optimal or not for $C_{1,r}$ class is still known.

\smallskip

In this section, we ask a question in a different direction:
what is the minimal spectral set for all operators in the classes $C_\alpha$ and $C_{1,r}$? Let us recall that a compact set $X \subset \C^n$ is called a minimal spectral set for a set $\mathcal T$ consisting of commuting $n$-tuples of Hilbert space operators if $X$ is a spectral set for every member of $\mathcal T$ and for any proper compact subset $X'$ of $X$, there is $\underline T \in \mathcal T$ such that $X'$ is not a spectral set for $\underline T$. Evidently, we are supposed to find a minimal compact set $X$ that contains $\CA_r$ as a proper subset and any operator in $C_\alpha$ or $C_{1,r}$ having its spectrum in $X$ satisfies von Neumann's inequality (\ref{eqn:new-001}) when $K=1$. So, we are ready to expand the underlying compact set from $\CA_r$ to something bigger but hold the constant $K$ fixed as $K=1$ in (\ref{eqn:new-001}). Since the operators in $C_{\alpha}$ or $C_{1,r}$ are contractions, it is evident that $\DC$ is a spectral set for them. Below, we show that no smaller compact subset of $\DC$ can be a spectral set for all operators in $C_\alpha$ or $C_{1,r}$, which essentially proves that $\overline{\D}$ is actually a minimal spectral set for both the classes. We also intend to find minimal spectral sets for the following classes of operators that are in the vicinity of $C_{\alpha}$ and $C_{1,r} \ $:
\begin{gather*}
C_\alpha^*=\{T: T^* \in C_\alpha\}, \ \ C_1 =\{T : T \ \text{is an invertible contraction} \}, \\
C_\beta =\{T : T \ \text{is an invertible operator}\}.
\end{gather*}
For an $\A_r$-contraction $T$, its adjoint  $T^*$ is also an $\A_r$-contraction. We have by Theorem \ref{thm_Bello_1.1} that both $T$ and $T^*$ are in $C_\alpha$. However, this inclusion is strict as was mentioned in Remark 3.5 of \cite{Bello}. Thus, we have the following strictly increasing chain of classes.

\begin{thm}\label{thm_include}
$\mathcal A_r \ \subsetneq  \ C_\alpha \ \subsetneq C_\alpha \cup C^*_\alpha \ \subsetneq \  C_{1, r} \subsetneq C_1 \subsetneq C_\beta$, where $\mathcal A_r=\{T: T \ \text{is an $\A_r$-contraction}\}$.
\end{thm}

One can easily extend Theorem \ref{thm_Bello_1.3} to the operators in $C_\alpha \cup C_\alpha^*$. Before proceeding further, we state a result from the literature that will be useful.

\begin{thm}[\cite{Paulsen}, Theorem 9.11 \& Corollary 9.12]\label{thm_Arv_Kset}
A compact set $X \subset \C$ is a complete $K$-spectral set for an operator $T$ acting on a Hilbert space $\HS$ if and only if there exists an invertible operator $S$ on $\HS$ such that $\|S^{-1}\|\|S\| \leq K$ and $X$ is a complete spectral set for $S^{-1}TS$.
\end{thm}

\begin{prop}\label{prop_CaCa*}
Let $T \in C_\alpha \cup C_\alpha^*$. Then $\CA_r$ is a complete $\sqrt{2}$-spectral set for $T$.
\end{prop}

\begin{proof}
	If $T \in C_\alpha$, then the conclusion follows from Theorem \ref{thm_Bello_1.3}. Let $T^* \in C_\alpha$. It follows from Theorem \ref{thm_Bello_1.3} that $\CA_r$ is a complete $\sqrt{2}$-spectral set for $T^*$. Hence, we have by Theorem \ref{thm_Arv_Kset} that there is an invertible operator $S$ such that $\|S^{-1}\| \|S\| \leq \sqrt{2}$ and $\CA_r$ is a complete spectral set for $S^{-1}T^*S$. Consequently, $\CA_r$ is a complete spectral set for $(S^{-1}T^*S)^*=S^*TS^{-*}$ and that $\|S^*\| \|S^{-*}\| \leq \sqrt{2}$. Again, by Theorem \ref{thm_Arv_Kset}, we have that $\CA_r$ is a complete $\sqrt{2}$-spectral set for $T$.
\end{proof}

We now show that the converse to Theorems \ref{thm_Bello_1.3} or \ref{thm_Crou_9} does not hold.

\begin{eg}
We refer to Example \ref{eg_5}. Consider the operator $T=M_z^t$ on $H^2(\Omega_r, \mathcal{E})$, where $\mathcal{E}$ is a Hilbert space. We have by Example \ref{eg_5} that $T \in C_\alpha$ and so, $T^* \in C_\alpha \cup C_\alpha^*$. By Proposition \ref{prop_CaCa*}, $\CA_r$ is a complete $\sqrt{2}$-spectral set for $T^*$. It follows from Example \ref{eg_5} that $T^* \notin C_\alpha$ and so, the converse of Theorem \ref{thm_Bello_1.3} does not hold. \qed
\end{eg}

\begin{eg}
For $0<r<w<1$, let 
\[
S=\begin{pmatrix}
1\slash \sqrt{2} & 0\\
0 & 1\\
\end{pmatrix} \quad  \text{and} \quad A=\begin{pmatrix}
w & (1-w^2)(w^2-r^2)\\
0 & w\\
\end{pmatrix}.
\]
We show that $A$ is an $\mathbb{A}_r$-contraction. To do so, we recall from \cite{Misra} the Hardy space $H^2(\partial \mathbb{A}_r)$ consisting of holomorphic functions from $L^2(\partial \A_r)$, where the inner product is given by
\begin{equation*}
\langle f,g\rangle=\frac{1}{2\pi}\overset{2\pi}{\underset{0}{\int}}f(e^{it})\overline{g(e^{it})}dt + \frac{1}{2\pi}\overset{2\pi}{\underset{0}{\int}}f(re^{it})\overline{g(re^{it})}dt\quad (f, g \in L^2(\partial \mathbb{A}_r)).
\end{equation*}
The space $H^2(\partial \mathbb{A}_r)$ is a reproducing kernel Hilbert space with kernel function given by $\displaystyle
\widehat{K}_r(\lambda, \mu)=\overset{\infty}{\underset{n=-\infty}{\sum}}\frac{(\lambda\overline{\mu})^n}{1+r^{2n}}$ for $\lambda, \mu \in \mathbb{A}_r$. Since $w \in (r, 1)$, we have
\begin{equation*}
\begin{split}
\widehat{K}_r(w,w)
=\overset{\infty}{\underset{n=0}{\sum}}\frac{w^{2n}}{1+r^{2n}} + \overset{\infty}{\underset{n=1}{\sum}}\frac{r^{2n}}{w^{2n}(1+r^{2n})}
\leq \overset{\infty}{\underset{n=0}{\sum}}w^{2n} + \overset{\infty}{\underset{n=1}{\sum}}(r\slash w)^{2n}
\leq \frac{1}{(1-w^2)(w^2-r^2)}.
\end{split}
\end{equation*}
It follows from Corollary $1.1'$ in \cite{Misra} that $A$ is an $\mathbb{A}_r$-contraction on $\mathbb{C}^2$. Evidently, $\|S\|\|S^{-1}\|=\sqrt{2}$. We have by Theorem \ref{thm_Arv_Kset} that $\CA_r$ is a complete $K$-spectral set for $T=S^{-1}AS$ for any $K \geq \sqrt{2}$.  For the sake of brevity, we let $a=(1-w^2)(w^2-r^2)$ and so, $T=\begin{pmatrix}
w & \sqrt{2} a \\
0 & w
\end{pmatrix}$. A tedious but simple computation gives
\[
\alpha(T^*, T)=\begin{pmatrix}
(1-w^2)(w^2-r^2) \ & \ -2\sqrt{2}aw^3+\sqrt{2}(1+r^2)aw \\
-2\sqrt{2}aw^3+\sqrt{2}(1+r^2)aw \ & \ -w^4-8a^2w^2+(1+r^2)(w^2+2a^2)-r^2
\end{pmatrix}
\] 
and
\[
\alpha(T, T^*)=\begin{pmatrix}
 -w^4-8a^2w^2+(1+r^2)(w^2+2a^2)-r^2 \ & \ -2\sqrt{2}aw^3+\sqrt{2}(1+r^2)aw \\
-2\sqrt{2}aw^3+\sqrt{2}(1+r^2)aw \ & \ (1-w^2)(w^2-r^2)
\end{pmatrix}.
\]
Moreover, we have that
\[
I-T^*T=\begin{pmatrix}
1-w^2 & -\sqrt{2}aw \\
-\sqrt{2}aw & 1-2a^2-w^2
\end{pmatrix} \quad \& \quad T^{-1}=w^{-2}VTV, \quad \text{where} \quad V=\begin{pmatrix}
1 & 0 \\
0 & -1
\end{pmatrix}.
\]
Since $1-w^2>0$, it follows that $I-T^*T \geq 0$ if and only if $\det(I-T^*T)=1-2a^2-2w^2+w^4 \geq 0$. Moreover, $\|rT^{-1}\|=rw^{-2}\|T\|$, because $V$ is a unitary. We now disprove converse of Theorem \ref{thm_Crou_9} and Proposition  \ref{prop_CaCa*} for suitable values of $r, w \in (0, 1)$ with $r<w$.

\begin{enumerate} 

\vspace{0.2cm}

\item For $r=0.35$ and $w=0.91$, we have that $a=0.12129264$ and $\det(I-T^*T)>0$. Hence, $I-T^*T \geq 0$ and so, $T$ is a contraction. Also, $rw^{-2}=3500\slash 8281<1$ and thus, $\|rT^{-1}\|=rw^{-2}\|T\| <\|T\| \leq 1$. Thus $T \in C_{1,r}$. It is not difficult to verify that $\det(\alpha(T^*, T))=\det(\alpha(T, T^*))<0$. Hence, $T \notin C_\alpha \cup C^*_\alpha$. Thus, converse to Proposition \ref{prop_CaCa*} does not hold.

\vspace{0.2cm}

\item Choose $r=0.52$ and $w=0.99$. Again, simple computations give $a=0.01412303$ and $\det(I-T^*T)<0$. Hence, $T \notin C_{1, r}$. Thus, the converse to Theorem \ref{thm_Crou_9} does not hold. \qed
\end{enumerate} 
\end{eg} 
We now search for minimal spectral sets for the following classes: $C_\alpha, C_\alpha \cup C^*_\alpha, C_{1,r}$, $C_1$. 

\smallskip

In general, a class $\mathcal{C}$ consisting of commuting $n$-tuple of operators need not have a minimal spectral set. For example, consider the family $\mathcal{C}_0=\{k I: k \in \mathbb{N}\}$, where $I$ is the identity operator on a Hilbert space. Evidently, no compact subset of $\C$ can be a spectral set for $\mathcal C_0$. Also, $C_{\beta}$ as mentioned above is a norm-unbounded family of operators and cannot have a minimal spectral set in $\C$. On the other hand, a norm-bounded family like the set of all contractions must have $\overline{\D}$ as a minimal spectral set and it follows from von Neumann's famous theorem \cite{v-N} and the fact that the spectrum of a pure isometry (i.e., the unilateral shift $M_z$) is equal to $\overline{\D}$. An analogous result holds for $\A_r$-contractions. Proposition 3.15 in \cite{PalII} tells us that the spectrum of a pure $\A_r$-isometry is $\CA_r$. Thus, $\CA_r$ is a minimal spectral set for all $\A_r$-contractions. However, we shall show that $\CA_r$ is not a minimal spectral set for any of the classes: $C_\alpha, C_\alpha \cup C_\alpha^*,C_{1,r}$. For proving this, we recall that an operator $T \in \mathcal{B}(\mathcal{H})$ is said to be \textit{completely non-normal} if there is no nonzero closed subspace $\mathcal{L}$ of $\HS$ that reduces $T$ and $T|_{\mathcal{L}}$ is normal.

\begin{thm} \label{thm:new-005}
$\DC$ is a minimal spectral set for $C_{1,r}$.
\end{thm}

\begin{proof}
Evidently for an operator $S \in C_{1,r}$, we have that $S$ and $rS^{-1}$ are contractions. Since every contraction has $\overline{\D}$ as a spectral set, it follows that $\overline{\D}$ is a spectral set for $C_{1,r}$. We now show that for any proper compact subset $Y$ of $\overline{\D}$, there is a member $T$ in $C_{1,r}$ such that $Y$ is not a spectral set for $T$. For this purpose we refer to an example due to G. Misra \cite{Misra}. For $0<r<1$, define 
	\[
	T=\begin{pmatrix}
		\sqrt{r} &  1-r\\
		0 & \sqrt{r} 
	\end{pmatrix} \quad \text{so that} \quad 
	T^*T=\begin{pmatrix}
		r & \sqrt{r}(1-r)\\
		\sqrt{r}(1-r) & r+(1-r)^2\\ 
	\end{pmatrix}. 
	\]
	It is easy to see that $\sigma(T)=\{\sqrt{r}\} \subseteq \A_r$ and $\|T\|=\|T^*T\|^{1\slash 2}=1$. Moreover, $T \in C_{1, r}$. Since $\|T\|=1$ and $T$ is completely non-normal, it follows from Corollary 1 of Section 5 in \cite{Williams} that no proper closed subset of $\overline{\mathbb{D}}$ is a spectral set for $T$. The proof is now complete.
\end{proof}

\begin{thm}\label{prop_min_Ca}
$\DC$ is a minimal spectral set for $C_\alpha$.
\end{thm}

\begin{proof}
By Theorem \ref{thm_include}, every operator in $C_\alpha$ is a contraction and so, $\DC$ is a spectral set for operators in $C_\alpha$. It suffices to show that there is an operator in $C_\alpha$ for which no proper closed subset of $\DC$ is a spectral set. We have shown in Example \ref{eg_5} that $M_z^t$ on $H^2(\Omega_r, \C)$ cannot be an $\A_s$-contraction for $0<s<1$ and $M_z^t \in C_\alpha$. Now we prove that no proper closed subset of $\overline{\D}$ is a spectral set for $M_z^t$.  For brevity, we denote by $e_n$ the vector in $H^2(\Omega_r, \C)$ with $1$ at $n$-th position and zero everywhere else.  Note that
\[
	\langle e_n, e_m \rangle 	= \left\{
				\begin{array}{ll}
					r^{2n}(1-r^{2})^{-1}, & m=n, n \geq 0\\
					(1-r^{2})^{-1}, & m=n,  n<0\\
					0, &  n \ne m\\
				\end{array} 
				\right. \quad \text{and} \quad \la f, e_n\ra = \left\{
				\begin{array}{ll}
					f_nr^{2n}(1-r^{2})^{-1}, & n \geq 0\\
					 f_n(1-r^{2})^{-1}, &  n <0 \\
				\end{array} 
				\right.
	\]
for all $f=\{f_n\}_{n \in \Z}$. Thus, $\{e_n: n \in \Z\}$ forms an orthogonal basis of $H^2(\Omega_r, \C)$. Let 
\[
w_n=\left\{
				\begin{array}{ll}
					e_nr^{-n}\sqrt{1-r^{2}}, & n \geq 0\\
					 e_n\sqrt{1-r^{2}}, &  n <0 \\
				\end{array} 
				\right. \quad \text{and} \quad \alpha_n=\left\{
				\begin{array}{ll}
					r, & n \geq 0\\
					 -1, &  n =-1 \\
					  1, &  n <-1. \\
				\end{array} 
				\right.
\]
The set $\{w_n : n \in \Z\}$ forms an orthonormal basis of $H^2(\Omega_r, \C)$ and $M_z^tw_n=\alpha_nw_{n+1}$. The adjoint is given by $M_z^{t*}w_n=\alpha_{n-1}w_{n-1}$ for all $n \in \Z$. Since $M_z^t \in C_\alpha$, we have by Theorem \ref{thm_include} that $M_z^t$ and $r(M_z^t)^{-1}$ are contractions. A routine calculation shows that $(M_z^t)^nw_{-n}=\alpha_{-n} \dotsc \alpha_{-1} w_0$ and $r^n(M_z^t)^{-n}w_{n+1}=w_1$ for all $n \in \mathbb{N}$. Therefore, $\|(M_z^t)^n\|=\|r^n(M_z^t)^{-n}\|=1$ for $n \in \mathbb{N}$. By Gelfand's formula, the spectral radii of $M_z^t$ and $r(M_z^t)^{-1}$ are given by 
\[
\rho(M_z^t)=\lim_{n \to\infty} \|(M_z^t)^n\|^{1\slash n}=1 \quad \text{and} \quad \rho((M_z^t)^{-1})=\lim_{n \to\infty} \|(M_z^t)^{-n}\|^{1\slash n}=1\slash r
\]
respectively. We have by Theorem 1.3.9 in \cite{Bourhim} that 
\[
\sigma(M_z^t)=\left\{z \in \C: \frac{1}{\rho((M_z^t)^{-1})} \leq |z| \leq \rho(M_z^t)\right\}=\CA_r.
\]
Thus, no compact subset of $\DC$ smaller than $\CA_r$ can be a spectral set for $M_z^t$. Let if possible, $X \subsetneq \DC$ be a spectral set for $M_z^t$. Then $\CA_r \subset X \subsetneq \DC$ and there exists $\lambda_0 \in \DC \setminus X$. Since $X$ is closed, there is an $\epsilon > 0$ such that $X \subseteq  \DC \setminus B(\lambda_0, \epsilon)$, where $B(\lambda_0, \epsilon)=\{z \in \C : |z-\lambda_0|<\epsilon\}$. Let $X_0=\DC \setminus B(\lambda_0, \epsilon)$. Therefore, $|\lambda_0|<r$ and the annular set $X_0$ is also a spectral set for $M_z^t$. One can choose $\epsilon>0$ sufficiently small such that $\epsilon<r-|\lambda_0|$. The map
\[
\phi: \DC\to \DC, \quad \phi(z)=\frac{z-\lambda_0}{1-\overline{\lambda}_0z}  \quad (z \in \DC)
\]
is an automorphism of $\DC$. Let $z \in X_0$. Then 
\[
|z-\lambda_0| \geq \epsilon, \quad  |1-\overline{\lambda}_0z| \leq 1+|\lambda_0||z| \leq 1+|\lambda_0| \quad \text{and so, } \quad |\phi(z)|=\frac{|z-\lambda_0|}{|1-\overline{\lambda}_0z|} \geq  \frac{\epsilon}{1+|\lambda_0|}.
\] 
Thus $\epsilon(1+|\lambda_0|)^{-1} \leq |\phi(z)| \leq 1$ and so, $\phi(X_0) \subseteq \CA_s$ for any $s<\epsilon(1+|\lambda_0|)^{-1}$. Putting everything together, we have that $\phi(M_z^t)=(M_z^t-\lambda_0)(I-\overline{\lambda}_0M_z^t)^{-1}$ is an $\A_s$-contraction. Thus, $S=\phi(M_z^t)^*$ is an $\A_s$-contraction. We have by Theorem \ref{thm_Bello_1.1} that 
\begin{equation}\label{eqn_min_Ca}
(1+s^2)-S^*S-s^2S^{-*}S^{-1} \geq 0 \quad \text{and so,} \quad (1+s^2)\|f\|^2-\|Sf\|^2-s^2\|S^{-1}f\|^2 \geq 0  
\end{equation}
for all $f \in H^2(\Omega_r, \C)$. It is not difficult to see that $(M_z^{t*})^nw_0=-w_{-n}$ and $(M_z^{t*})^{-n}w_0=r^{-n}w_{n}$ for all $n \in \mathbb{N}$. Again, by routine computations we have that

\begin{align*}
	Sw_0
	=\left[\overset{\infty}{\underset{n=0}{\sum}} \lambda_0^n(M_z^{t*})^{n+1}-\overline{\lambda}_0\overset{\infty}{\underset{n=0}{\sum}} \lambda_0^n(M_z^{t*})^{n}\right]w_0
	=-(1-|\lambda_0|^2)\overset{\infty}{\underset{n=1}{\sum}}\lambda_0^{n-1}w_{-n}-\overline{\lambda}_0w_0
\end{align*}
and
\begin{align*}		
	S^{-1}w_0
	=(I-\lambda_0M_z^{t*})(M_z^{t*})^{-1}\left(I-\overline{\lambda}_0(M_z^{t*})^{-1}\right)^{-1}w_0
	&=\left[\overset{\infty}{\underset{n=0}{\sum}} \overline{\lambda}_0^n(M_z^{t*})^{-n-1}-\lambda_0\overset{\infty}{\underset{n=0}{\sum}}\overline{\lambda}_0^n(M_z^{t*})^{-n}\right]w_0\\
	&=(1-|\lambda_0|^2)\overset{\infty}{\underset{n=1}{\sum}} \overline{\lambda}_0^{n-1}r^{-n}w_{n}-\lambda_0w_0.
\end{align*}
Consequently, we have that 
\[
\|Sw_0\|^2=|\lambda_0|^2+(1-|\lambda_0|^2)^2 \overset{\infty}{\underset{n=1}{\sum}}|\lambda_0|^{2(n-1)}=|\lambda_0|^2+(1-|\lambda_0|^2)^2\frac{1}{1-|\lambda_0|^2}=1
\]
and
\[
\|S^{-1}w_0\|^2=|\lambda_0|^2+(1-|\lambda_0|^2)^2 \overset{\infty}{\underset{n=1}{\sum}}\frac{|\lambda_0|^{2(n-1)}}{r^{2n}}=|\lambda_0|^2+\frac{(1-|\lambda_0|^2)^2}{r^2-|\lambda_0|^2}.
\]
Hence, it follows that
\begin{equation*}
\begin{split}
(1+s^2)\|w_0\|^2-\|Sw_0\|^2-s^2\|S^{-1}w_0\|^2
&=(1+s^2)-1-s^2\left(|\lambda_0|^2+\frac{(1-|\lambda_0|^2)^2}{r^2-|\lambda_0|^2}\right)\\
&=\frac{s^2(1-|\lambda_0|^2)}{r^2-|\lambda_0|^2}(r^2-1)<0,
\end{split}
\end{equation*}
which contradicts (\ref{eqn_min_Ca}). Hence, $X$ cannot be a spectral set for $M_z^t$. The proof is now complete.
\end{proof}

A consequence of Theorem \ref{thm_include} and Theorem \ref{prop_min_Ca} is the following result. 

\begin{thm}\label{cor_min_3}
$\DC$ is a minimal spectral set for $C_\alpha^*, \,C_\alpha \cup C_\alpha^*$ and $C_{1}$.
\end{thm}

\begin{proof}
Since $C_\alpha^*, C_\alpha \cup C_\alpha^*, C_1$ consist of contractions, it follows that $\DC$ is a spectral set for each one of these classes. We have by Theorem \ref{thm_include} that $C_\alpha \subseteq C_\alpha \cup C_\alpha^* \subseteq  C_1$. By Theorem \ref{prop_min_Ca}, $\DC$ is a minimal spectral set for $C_\alpha$ and so, $\DC$ is a minimal spectral set for $C_\alpha \cup C_\alpha^*$ and $C_{1}$. We now show that the operator $(M_z^t)^*$ on $H^2(\Omega_r, \C)$, as in Example \ref{eg_5}, belongs to $C_\alpha^*$ with $\DC$ as a minimal spectral set. Again by Example \ref{eg_5}, the operator $M_z^t$ on $H^2(\Omega_r, \C)$ belongs to $C_\alpha$ and so, $(M_z^t)^* \in C_\alpha^*$. Following the proof of Theorem \ref{prop_min_Ca}, we have that $\sigma(M_z^t)=\CA_r$ and so, $\sigma((M_z^t)^*)=\CA_r$.

\smallskip

We apply similar computations and techniques as in Theorem \ref{prop_min_Ca} to show that $\DC$ is a minimal spectral set for $(M_z^t)^*$. To do so, we follow the same notations  as in Theorem \ref{prop_min_Ca}. Let if possible, $X \subsetneq \DC$ be a spectral set for $(M_z^t)^*$. Then $\CA_r \subset X \subsetneq \DC$ and choose $\lambda_0 \in \DC \setminus X$. As $X$ is closed, one can choose $\epsilon > 0$ such that $X \subseteq  \DC \setminus B(\lambda_0, \epsilon)$. Let $X_0=\DC \setminus B(\lambda_0, \epsilon)$. Therefore, $|\lambda_0|<r$ and $X_0$ is also a spectral set for $(M_z^t)^*$. Choose $\epsilon>0$ small enough so that $\epsilon<r-|\lambda_0|$. As shown in the proof of Theorem \ref{prop_min_Ca}, the automorphism of $\DC$ given by
\[
\phi: \DC\to \DC, \quad \phi(z)=\frac{z-\lambda_0}{1-\overline{\lambda}_0z}  \quad (z \in \DC)
\]
satisfies $\epsilon(1+|\lambda_0|)^{-1} \leq |\phi(z)| \leq 1$ and $\phi(X_0) \subseteq \CA_s$ for $0<s<\epsilon(1+|\lambda_0|)^{-1}$. Let $P=\phi((M_z^t)^*)$. Since $(M_z^t)^*$ has $X_0$ as a spectral set, $P$ has $\phi(X_0)$ as a spectral set. Consequently, $P=((M_z^t)^*-\lambda_0)(I-\overline{\lambda}_0(M_z^t)^*)^{-1}$ is an $\A_s$-contraction. We have by Theorem \ref{thm_Bello_1.1} that
\begin{equation}\label{eqn_min_CaII}
(1+s^2)\|f\|^2-\|Pf\|^2-s^2\|P^{-1}f\|^2 \geq 0  
\end{equation}
for all $f \in H^2(\Omega_r, \C)$. Similar calculations as in the proof of Theorem \ref{prop_min_Ca} give that  
\begin{align*}
Pw_0=-\lambda_0w_0-(1-|\lambda_0|^2)\overset{\infty}{\underset{n=1}{\sum}}\overline{\lambda_0}^{n-1}w_{-n} \quad \& \quad P^{-1}w_0=-\overline{\lambda}_0w_0+(1-|\lambda_0|^2)\overset{\infty}{\underset{n=1}{\sum}}\lambda_0^{n-1}r^{-n}w_{n},
\end{align*}
where $\{w_n : n \in \Z\}$ forms an orthonormal basis of $H^2(\Omega_r, \C)$. Consequently, $\|Pw_0\|^2=1$ and $\|P^{-1}w_0\|^2=|\lambda_0|^2+(1-|\lambda_0|^2)^2(r^2-|\lambda_0|^2)^{-1}$. Then
\begin{equation*}
\begin{split}
(1+s^2)\|w_0\|^2-\|Pw_0\|^2-s^2\|P^{-1}w_0\|^2=\frac{s^2(1-|\lambda_0|^2)}{r^2-|\lambda_0|^2}(r^2-1)<0,
\end{split}
\end{equation*}
which contradicts (\ref{eqn_min_CaII}). Therefore, $X$ cannot be a spectral set for $(M_z^t)^*$. This finishes the proof.
\end{proof}

We now turn our attention to find minimal spectral sets for the induced pairs $\kappa(T)$, when $T$ belongs to the set of all $\A_r$-contractions.

\begin{thm}
$Z(q) \cap \BC$ is a minimal spectral set for $\mathcal{F}=\{\kappa(T) : T \ \text{is an $\A_r$-contraction}\}$, where $\displaystyle q(z_1, z_2)=z_1z_2-\frac{r}{1+r^2}.$ 
\end{thm}

\begin{proof}
Let $T$ be an $\A_r$-contraction. Then $\CA_r$ is a spectral set for $T$ and by Theorem \ref{thm_A_con_BallII}, the variety $\kappa(\CA_r)=Z(q) \cap \BC$ is a spectral set for every operator in $\mathcal{F}$. To see the minimality, let $C$ be a spectral set for all operators in $\mathcal{F}$. Let $V$ be a pure $\A_r$-isometry. It follows from Proposition 3.15 in \cite{PalII} that $\sigma(V)=\CA_r$. By spectral mapping theorem and Lemma \ref{lem_pi}, $\sigma_T(\kappa(V))=\kappa(\sigma(V))=\kappa(\CA_r)=Z(q) \cap \BC$ and so, $Z(q) \cap \BC \subseteq C$. Hence, $Z(q) \cap \BC$ is a minimal spectral set for $\mathcal{F}$.
\end{proof}

With the same spirit, we want to find minimal spectral sets for the classes of operator pairs induced by $C_{\alpha}, C_{\alpha}\cup C_{\alpha}^*, C_{1,r}, C_1$ and $C_\beta$ under the map $\kappa$. For the first three classes, we have by Theorem \ref{prop_min_Ca} and Theorem \ref{cor_min_3} that $\DC$ is a minimal spectral set for $C_\alpha$, $C_\alpha \cup C_\alpha^*$ and $C_{1, r}$. Since $\kappa$ is well-defined on $\mathbb{C} \setminus \{0\}$ and is unbounded on $\overline{\D}\setminus \{0\}$, it turns out that the notion of minimal spectral sets does not work for these classes of induced pairs. So, we consider the more general notion of von Neumann set and minimal von Neumann set. Recall from Definition \ref{defn:new-001} that a set $Y \subseteq \C^n$ is said to be a von Neumann set for a commuting tuple of operators $\underline{T}=(T_1, \dotsc, T_n)$ if $\sigma_T(\underline{T}) \subseteq Y$ and $\|f(\underline{T})\| \leq \|f\|_{\infty, Y}$ for every $f \in Rat(Y)$. Also, for a class of commuting $n$-tuples of operators $\mathcal{C}$, a (closed) set $Y \subset \C^n$ is called a minimal $($closed$)$ von Neumann set for $\mathcal{C}$ if $Y$ is a von Neumann set for every operator tuple in $\mathcal{C}$ and for any proper (closed) subset $Y'$ of $Y$, there is $\underline{S}$ in $\mathcal{C}$ such that $Y'$ is not a von Neumann set for $\underline S$. We say that $Y$ is a minimal von Neumann set for $\underline{S}$ if $Y$ is a minimal von Neumann set for the set $\mathcal{C}=\{\underline{S}\}$.
The notion of von Neumann set is useful as the following result shows. 

\begin{lem}\label{lem_invc_vN}
$T$ is an invertible contraction if and only if $\DC \setminus \{0\}$ is a von Neumann set for $T$. 
\end{lem}

\begin{proof}
Let $T$ be an invertible contraction. In this case, $\sigma(T) \subseteq \DC \setminus \{0\}$. Let $f \in Rat(\DC \setminus \{0\})$. If $f$ has no pole at $z=0$, then $f \in Rat(\DC)$. By maximum modulus principle, we have that
\[
\|f(T)\| \leq \|f\|_{\infty, \DC}=\|f\|_{\infty, \T} \leq \|f\|_{\infty, \DC \setminus \{0\}}.
\]
Let $f$ have a pole at $z=0$ of order $m$. We can write $f(z)=z^{-m}g(z)$ for some $g \in Rat(\DC)$ with $g(0) \neq 0$. Consequently, $\|f(T) \| \leq \|f\|_{\infty, \DC \setminus \{0\}}=\infty$ and thus, $\DC \setminus \{0\}$ is a von Neumann set for $T$. The converse follows directly from the fact that $\DC$ is a spectral set for any contraction $T$. 
\end{proof}

Let $\mathcal{C}$ be any one of the classes mentioned in Theorem \ref{thm_include}. The next theorem gives a connection between von Neumann sets for operators in $\mathcal{C}$ and $\kappa(\mathcal{C})$, where $\kappa(\mathcal C)=\{ \kappa(T): \ T \in \mathcal C \}$.

\begin{thm}\label{thm_vN_iff}
Let $T$ be an invertible operator and let $Y \subseteq \C \setminus \{0\}$. Then $Y$ is a von Neumann set for  $T$  if and only if $\kappa(Y)$ is a von Neumann set for $\kappa(T)$. 
\end{thm}

\begin{proof}
Let $Y$ be a von Neumann set for $T$. By spectral mapping theorem, $\sigma_T(\kappa(T))=\kappa(\sigma(T)) \subseteq \kappa(Y)$. Let $f \in Rat(\kappa(Y))$. Then $f \circ \kappa \in Rat(Y)$ and we have that
\[
\|f(\kappa(T))\| \leq \sup\{|f(\kappa(y))|: y \in Y\} \leq \sup\{|f(\kappa(y))|: \kappa(y) \in \kappa(Y)\} \leq \|f\|_{\infty, \kappa(Y)}.
\]
Conversely, assume that $\kappa(Y)$ is a von Neumann set for $\kappa(T)$. Let $\lambda \in \sigma(T)$. Again by spectral mapping theorem, $\kappa(\lambda) \in \sigma_T(\kappa(T)) \subseteq \kappa(Y)$. Since $\kappa$ is injective, we have that $\lambda \in Y$ and so, $\sigma(T) \subseteq Y$. It remains to show the von Neumann's inequality for $T$. Let $f \in Rat(Y)$. It is not difficult to see that $\sqrt{1+r^2}z_1 \in Y$ for all $(z_1, z_2) \in \kappa(Y)$. Define $F(z_1, z_2)=f(\sqrt{1+r^2}z_1)$ which is in $Rat(\kappa(Y))$.  Then
\begin{equation*}
\begin{split}
\|f(T)\| & =\|F(\kappa(T))\| \\
&=\sup\{|F(z_1, z_2)| : (z_1, z_2) \in \kappa(Y) \}\\
&\leq \sup\{|f(\sqrt{1+r^2}z_1)| : (z_1, z_2) \in \kappa(Y) \}\\
& \leq \sup\{|f(\sqrt{1+r^2}z_1)| : \sqrt{1+r^2}z_1 \in Y \} \quad [\text{as} \ \text{$\sqrt{1+r^2}z_1 \in Y$ for $(z_1, z_2) \in \kappa(Y)$}]\\
& \leq \|f\|_{\infty, Y}.\\
\end{split}
\end{equation*}
Therefore, $Y$ is a von Neumann set for $T$ which completes the proof.
\end{proof}

\begin{thm}\label{thm_min_vN}
$\kappa(\DC \setminus \{0\})$ is a minimal closed von Neumann set for each of the following classes: 
\[
\kappa(C_\alpha), \quad \kappa(C_\alpha^*), \quad \kappa(C_\alpha \cup C_\alpha^*), \quad \kappa(C_{1,r}) \quad \kappa(C_1).
\]
\end{thm}

\begin{proof}
Since $\DC \setminus \{0\}$ is closed in $\C \setminus \{0\}$, it follows from Lemma \ref{lem_pi} that $\kappa(\DC \setminus \{0\})$ is closed in $\C^2$. Let $\mathcal{C}$ be any one of the classes $C_\alpha$, $C_\alpha^*$, $C_\alpha \cup C_\alpha^*, C_{1,r}, C_1$ and let $X_*=\kappa(\DC \setminus \{0\})$. Let $T \in \mathcal{C}$. We have by Theorem \ref{thm_include} that $T$ is an invertible contraction. By Lemma \ref{lem_invc_vN}, $\DC \setminus \{0\}$ is a von Neumann set for $T$. Finally, it follows from Theorem \ref{thm_vN_iff} that $X_*$ is a von Neumann set for $\kappa(T)$. For the minimality, we show that no proper closed subset of $X_*$ is a von Neumann set for $\kappa(\mathcal{C})$. Let $C \subsetneq  X_*$ be a closed von Neumann set for $\kappa(\mathcal{C})$. Then there exists $\alpha_0=(\alpha_{01}, \alpha_{02}) \in X_*$ such that $\alpha_0 \notin C$. Consequently, there exists a unique $z_0 \in \DC \setminus\{0\}$ such that
\[
\alpha_0=(\alpha_{01}, \alpha_{02})=\kappa(z_0)=\left(\frac{z_0}{\sqrt{1+r^2}}, \frac{rz_0^{-1}}{\sqrt{1+r^2}} \right).
\]
If $|z_0| \geq r$, then $z_0 \in \CA_r$. Then the operator $T_0=z_0I$ is in $\mathcal{C}$ and $\sigma(T_0)=\{z_0\}$. By spectral mapping theorem, $\sigma_T(\kappa(T_0))=\{\alpha_0\}$ and so, $\alpha_0 \in C$, which gives a contradiction. Hence, $|z_0|<r$. Then 
\[
|\alpha_{01}|=\frac{|z_0|}{\sqrt{1+r^2}} <\frac{r}{\sqrt{1+r^2}}.
\]
Since $C$ is closed and $\alpha_0 \notin C$, there exists $\epsilon>0$ sufficiently small such that 
\[
0< \epsilon <|\alpha_{01}|, \quad 0<\epsilon<\frac{r}{\sqrt{1+r^2}}-|\alpha_{01}| \quad \text{and} \quad C \subseteq X_0=X_* \setminus \left(B(\alpha_{01}, \epsilon) \times B(\alpha_{02}, \epsilon) \right),
\]
where $B(\alpha_{0j}, \epsilon)=\{z: |z-\alpha_{0j}|<\epsilon\}$  for $j=1,2$.  Evidently, $X_0$ is also a closed von Neumann set for $\kappa(\mathcal{C})$. Following Example \ref{eg_5}, we have that $M_z^t$ on $H^2(\Omega_r, \C)$ is in $C_\alpha$. We denote by $S$ the operator $M_z^t$, when we assume $\mathcal C$ to be any of the classes $C_{\alpha}, \ C_{\alpha}\cup C_{\alpha}^*, \ C_{1,r}, \ C_1$ and we also use the same notation $S$ for $(M_z^t)^*$, when considering $\mathcal C$ to be the class $C_{\alpha}^*$. We consider the scalars given by
\[
a=\frac{1}{\sqrt{1+r^2}} \quad \text{and} \quad  \delta=\min\left\{\epsilon, \epsilon \left(\frac{1+r^2}{r}\right)|\alpha_{01}|(|\alpha_{01}|-\epsilon)\right\}.
\]
Let $(S_1, S_2)=\kappa(S)$. Then $S_1=aS$ and $S_2=arS^{-1}$. Following the proof of Theorems \ref{prop_min_Ca} \& \ref{cor_min_3}, we have that $\DC$ is a minimal spectral set for $S$ , $\sigma(S) =\CA_r$ and $\|S\|=1$. Thus, $a\DC$ is a minimal spectral set for $S_1, \sigma(S_1)=\{az: z \in \CA_r\}$ and $\|S_1\|=a$. Let us define
\[
Y_0=\{z \in \C: |z| \leq a, |z-\alpha_{01}| \geq \delta \},
\]
which is a compact set. Let $w=az \in \sigma(S_1)$ for some $z \in \CA_r$. Then, $|w| \leq a$ and
\[
 |w-\alpha_{01}|=|az-\alpha_{01}| \geq |a||z|-|\alpha_{01}| \geq r|a|-|\alpha_{01}|=\frac{r}{\sqrt{1+r^2}} -|\alpha_{01}| \geq \epsilon \geq \delta
\]
and so, $w \in Y_0$. Thus, $\sigma(S_1) \subseteq Y_0$. We show that if $(z_1, z_2) \in X_0$, then $z_1 \in Y_0$. Take $(z_1, z_2) \in X_0$ which is a subset of $X_*$. One can choose $z \in \DC \setminus \{0\}$ such that $(z_1, z_2)=(az, arz^{-1})$ and at least one of the following holds: 
\[
|z_1-\alpha_{01}| \geq \epsilon \quad \text{or} \quad |z_1-\alpha_{01}| < \epsilon \leq  |z_2-\alpha_{02}|.
\]
Note that $|z_1| \leq a$. If $|z_1-\alpha_{01}| \geq \epsilon$, then $z_1 \in Y_0$ since $\delta \leq \epsilon$. Now, let $|z_1-\alpha_{01}| < \epsilon \leq  |z_2-\alpha_{02}|$. Then $|\alpha_{01}|-|z_1| \leq |z_1-\alpha_{01}| <\epsilon$ and 
\[
\epsilon \leq |z_2-\alpha_{02}|=\left| \frac{rz_1^{-1}}{1+r^2}- \frac{r\alpha_{01}^{-1}}{1+r^2} \right|=\left(\frac{r}{1+r^2}\right)|z_1^{-1}-\alpha_{01}^{-1}|=\left(\frac{r}{1+r^2}\right)\frac{|z_1-\alpha_{01}|}{|z_1\alpha_{01}|}.
\]
Consequently, we have that
\begin{equation*}
\begin{split}
|\alpha_{01}-z_1|
=\frac{|\alpha_{01}-z_1|}{|\alpha_{01}||z_1|}|\alpha_{01}||z_1|
 \geq \epsilon\left(\frac{1+r^2}{r}\right)|\alpha_{01}|(|\alpha_{01}|-\epsilon)
\geq \delta 
\end{split}
\end{equation*}
and thus, $z_1 \in Y_0$. Let $f \in Rat(Y_0)$. Then $g(z_1, z_2)=f(z_1)$ is in $Rat(X_0)$. Since $X_0$ is a von Neumann set for $\kappa(S)=(S_1, S_2)$, we have that
\[
\|f(S_1)\|=\|g(S_1, S_2)\| \leq \sup\{|g(z_1, z_2)| : (z_1, z_2) \in X_0  \} \leq \{ |f(z_1)| : z_1 \in Y_0\} =\|f\|_{\infty, Y_0}.
\] 
Hence, $Y_0 \subsetneq a\DC$ is a spectral set for $S_1$ which is a contradiction. Thus $X_*=\kappa(\DC \setminus \{0\})$ is a minimal closed von Neumann set for $\kappa(\mathcal{C})$ and the proof is now complete. 
\end{proof}

We conclude this section by providing a minimal closed von Neumann set for $\kappa(C_\beta)$.

\begin{thm}
$Z(q)$ is a minimal closed von Neumann set for $\kappa(C_\beta)$, where $q(z_1, z_2)=z_1z_2-\frac{r}{1+r^2}.$
\end{thm}

\begin{proof}
Let $T \in C_\beta$. Following the proof of Lemma \ref{lem_invc_vN}, we have that 
\[
 \|T\|\DC \setminus \{0\} = \{z\in \C : 0< |z| \leq \|T\|\}
 \]
is a von Neumann set for $T$. By Theorem \ref{thm_vN_iff}, $\kappa(\|T\|\DC \setminus \{0\})$ is  a von Neumann set for $\kappa(T)$, implying that $Z(q)$ is a von Neumann set for $\kappa(T)$ as $\kappa(\|T\|\DC \setminus \{0\}) \subset Z(q)$. Thus $Z(q)$ is a closed von Neumann set for operators in $\kappa(C_\beta)$. To see the minimality, let $C$ be a von Neumann set for all operators in $\kappa(C_\beta)$. Let $\alpha \in Z(q)$. One can find $z \in \C \setminus \{0\}$ such that $\alpha=\kappa(z)$. Define $T=zI$, where $I$ is an identity operator on a Hilbert space. Then $T \in C_\beta$ and by spectral mapping principle, $\sigma_T(\kappa(T))=\{\alpha\}$. Thus, $\alpha \in C$ and so, $Z(q) \subseteq C$. This shows that no proper subset of $Z(q)$ can be a von Neumann set for all operators in $\kappa(C_\beta)$, which completes the proof.
\end{proof}

		\section{The quantum annulus}\label{sec_quant}	
		
		\vspace{0.2cm}
		
		\noindent McCullough and Pascoe \cite{Pas-McCull} considered a different annulus $A_r=\{z \in \C : r<|z|<r^{-1}\}$ for $r \in (0,1)$ and introduced the following classes of operators related to $A_r$:
		\begin{enumerate}
			\item $SA_r=\{T: \overline{A}_r  \ \text{is a spectral set for} \ T \}$;
			\item $PA_r=\{T\,:\, T \text{ is an invertible operator and } \, r^2+r^{-2}-T^*T-T^{-*}T^{-1} \geq 0\}$;
			\item $QA_r = \{T\,:\, T \text{ is an invertible operator and } \, \|T\|, \, \|T^{-1}\| \leq r^{-1}  \}$.
		\end{enumerate}
		We refer to $QA_r$ as the \textit{quantum annulus} and the operators in $SA_r$ are referred to as \textit{$A_r$-contractions}. The following result gives a chain of these classes. 
		
		\begin{thm}[\cite{Pas-McCull}, Section 1]
			$SA_r \subsetneq PA_r \subsetneq QA_r$.
		\end{thm}
		
		The classes $SA_r, PA_r$ and $QA_r$ are analogues of $\{T: T \ \text{is an $A_r$-contraction}\},\, C_\alpha$ and $C_{1, r}$ respectively for the annulus $A_r$. It turns out that these classes are comparable with their respective analogs. To see this, note that the map 
		given by 
		\[
		\varphi: A_r \to \A_{r^2}, \quad \varphi(z)=rz
		\]
		defines a biholomorphism with $\varphi^{-1}(z)=r^{-1}z$. To explain the results more clearly, we use the following notations for the rest of this section.
		\[
		\alpha_r(T^*, T)=-T^{*2}T^2+(1+r^2)T^*T-r^2I \quad \text{and} \quad C_\alpha(r)=\{T: T \ \text{is invertible and} \ \alpha_r(T^*, T) \geq 0\},
		\]
		which was denoted by $\alpha(T^*, T)$ and $C_\alpha$ respectively in the previous sections. One can easily establish the following result using the biholomorphism $\varphi$ from $A_r$ to $\A_{r^2}$.
		
		\begin{lem}\label{lem_quantI}
			Let $T$ be an invertible operator. Then 
			\begin{enumerate}
				\item $T$ is an $\A_r$-contraction if and only if $r^{-1\slash 2} T \in SA_{\sqrt{r}}$. Also, $T \in SA_r$ if and only if $rT$ is an $\A_{r^2}$-contraction.
				\item $T \in C_\alpha(r)$ if and only if $r^{-1\slash 2} T \in PA_{\sqrt{r}}$. Also, $T \in PA_r$ if and only if $rT \in C_\alpha(r^2)$.
				\item $T \in C_{1, r}$ if and only if $r^{-1\slash 2} T \in QA_{\sqrt{r}}$. Also, $T \in QA_r$ if and only if $rT \in C_{1, r^2}$.
			\end{enumerate}
			
		\end{lem}
		
		Consequently, all the results from the previous sections for $\A_r$-contractions, $C_\alpha(r)$ and $C_{1, r}$ have analogues for $SA_r, PA_r$ and $QA_r$ respectively. We also define an analogue of the map $\kappa$ which is given by
		\[
		\kappa_0: \C \setminus \{0\} \to \C^2, \quad \kappa_0(z)=\left(\frac{z}{\sqrt{r^{-2}+r^2}}, \frac{z^{-1}}{\sqrt{r^{-2}+r^2}} \right).
		\]
		
		Note that 
		$
		\kappa_0(\overline{A}_r)=Z(q_0) \cap \BC$, where $q_0(z_1, z_2)=z_1z_2-\frac{1}{r^2+r^{-2}}$. Below, we have an analogue of Theorem \ref{thm_A_con_BallII} in this setting.
		
		\begin{thm}
			An invertible operator $T$ is in $SA_r$ if and only if $Z(q_0) \cap \BC$ is a complete spectral set for $\kappa_0(T)$.
		\end{thm}		
		
One can imitate the proof of Theorem \ref{thm_A_con_BallII} to frame a proof to the above theorem.	We conclude this section by finding minimal closed von Neumann sets for $\kappa_0(PA_r)$ and $\kappa_0(QA_r)$. Recall that for a class of invertible operators $\mathcal{C}$, $\kappa_0(\mathcal{C})=\{\kappa_0(T) : T \in \mathcal{C}\}$. 
		
		\begin{prop}\label{prop_quantII}
			$r^{-1}\DC$ is a minimal spectral for $PA_r$ and $QA_r$.
		\end{prop}	
		
		\begin{proof}
			Consider the biholomorphism $\Phi: r^{-1}\DC \to  \DC$ given by $\Phi(z)=rz$. It is evident that a compact set $X \subseteq \DC$ is a spectral set for an operator $A$ if and only if $\Phi^{-1}(X)=r^{-1}X \subseteq r^{-1}\DC$ is a spectral set for $r^{-1}A$. Consequently, we have by Lemma \ref{lem_quantI} that $X \subseteq r^{-1}\DC$ is a spectral set for $PA_r$ and $QA_r$ if and only if $\Phi(X)=rX$ is a spectral set for $C_\alpha(r^2)$ and $C_{1, r^2}$ respectively. It follows from Theorems \ref{thm:new-005} \& \ref{prop_min_Ca} that $\DC$ is a minimal spectral set for $C_\alpha(r^2)$ and $C_{1, r^2}$. Therefore, $r^{-1}\DC$ is a minimal spectral set for $PA_r$ and $QA_r$. This finishes the proof.
		\end{proof}

		\begin{thm}
			$\kappa_0(r^{-1}\DC\setminus\{0\})$ is a minimal closed von Neumann set for $\kappa_0(PA_r)$ and $\kappa_0(QA_r)$.
		\end{thm}		
		
		\begin{proof}
			Let $T \in PA_r$ or $QA_r$. We have by Proposition \ref{prop_quantII} that $r^{-1}\DC$ is a spectral set for $T$. Following the proof of Lemma \ref{lem_invc_vN}, one can prove that $Y=r^{-1}\DC\setminus \{0\}$ is a von Neumann set for $T$. Evidently, $\kappa_0(Y)$ is a closed subset of $\C^2$. By spectral mapping theorem, $\sigma_T(\kappa_0(T))=\kappa_0(\sigma(T)) \subseteq \kappa_0(Y)$. Let $f \in Rat(\kappa_0(Y))$. Then $f \circ \kappa_0 \in Rat(Y)$ and we have that
			\[
			\|f(\kappa_0(T))\| \leq \sup\{|f(\kappa_0(y))|: y \in Y\} \leq \sup\{|f(\kappa_0(y))|: \kappa_0(y) \in \kappa_0(Y)\} \leq \|f\|_{\infty, \kappa_0(Y)}.
			\]
			Thus, $\kappa_0(Y)$ is a closed von Neumann set for $\kappa_0(PA_r)$ and $\kappa_0(QA_r)$. By Example \ref{eg_5} and following the proof of Theorem \ref{prop_min_Ca}, we have that $M_z^t$ on $H^2(\Omega_{r^2}, \C)$ is in $C_\alpha(r^2)$ and has $\DC$ as a minimal spectral set. It follows from Lemma \ref{lem_quantI} that $T=r^{-1}M_z^t \in PA_r$ and $r^{-1} \DC$ is a minimal spectral for $T$. One can follow the proof of Theorem \ref{thm_min_vN} to show that $\kappa_0(Y)$ is a minimal closed von Neumann set for $\kappa_0(T)$ and the desired conclusion follows.
		\end{proof}	
		
		\noindent \textbf{Funding.} The first named author was supported in part by the `Core Research Grant' with Award No. CRG/2023/005223 of Anusandhan National Research Foundation (ANRF), Govt. of India. The second named author was supported by the Prime Minister's Research Fellowship with Award No. PMRF ID 1300140 of Govt. of India.
		
		\medskip
		
		\noindent \textbf{Declarations.} No data was analysed or used during the course of the paper. All authors of this paper contributed equally to article. Also, there is no competing interest.

\end{document}